\documentclass[10pt]{amsart} 

\usepackage{amssymb}
\usepackage{mathrsfs}
\usepackage{tikz}
\usepackage{pgfplots} 
\usepackage{nicefrac}
\usetikzlibrary{arrows}
\usepackage[linktocpage=true]{hyperref} 

\theoremstyle{plain}

\newtheorem{thm}{Theorem}[section]
\newtheorem{cor}[thm]{Corollary}
\newtheorem{lem}[thm]{Lemma}
\newtheorem{prop}[thm]{Proposition}

\theoremstyle{definition}
\newtheorem{defn}[thm]{Definition}
\newtheorem{sit}[thm]{Situation}
\newtheorem{claim}[thm]{Claim}
\newtheorem{Q}[thm]{Question}

\theoremstyle{remark}
\newtheorem{rmk}[thm]{Remark}

\newcommand{\zz}{{\mathbb Z}}

\newcommand{\cc}{{\mathbb C}}
\newcommand{\kk}{{\mathbb K}}
\newcommand{\hh}{{\mathrm H}}
\newcommand{\spec}{\mathrm{Spec}\mathop{ }}
\newcommand{\modu}{\mathcal}
\newcommand{\sh}{\mathscr}

\hypersetup{colorlinks=true, citecolor=blue, linkcolor=blue, urlcolor=blue}

\begin{document}

\title[The Severi problem for abelian surfaces in the primitive case]{The Severi problem for abelian surfaces in the primitive case}

\author[Adrian Zahariuc]{Adrian Zahariuc}
\address{Mathematics and Statistics Department, University of Windsor, 401 Sunset Ave, Windsor, ON, N9B 3P4, Canada}
\email{adrian.zahariuc@uwindsor.ca}
\thanks{The author was partially supported by an NSERC Discovery Grant}



\begin{abstract}
We prove that the irreducible components of primitive class Severi varieties of general abelian surfaces are completely determined by the maximal factorization through an isogeny of the maps from the normalized curves.
\end{abstract}

\maketitle

\setcounter{tocdepth}{1}  
\tableofcontents 

\section{Introduction}\label{sec: intro} 

We work over the field $\cc$ of complex numbers.

Severi varieties are roughly moduli spaces of plane curves of fixed degree and geometric genus \cite[Anhang F]{[Se21]}. More generally, the moduli spaces of curves of fixed homology class and geometric genus on a projective surface can also be called Severi varieties. Although the local properties of Severi varieties: dimension, smoothness, facts concerning the general curves, etc., are reasonably well-understood for a few types of surfaces, not much is known about the global geometry of these spaces besides enumerative aspects. By a landmark theorem of Harris \cite{[Ha86]}, the Severi varieties of ${\mathbb P}^2$ are irreducible. This was previously known as the Severi problem; by extension, proving that the Severi varieties of certain surfaces are irreducible, or identifying their irreducible components, may be referred to as `a' Severi problem.

The Severi problem for K3 surfaces is a well-known folklore conjecture. With the trivial exception of the genus $0$ case, the Severi varieties of (say, very general) projective K3 surfaces are expected to be irreducible. The problem is still open even for curves in the primitive class (that is, the positive generator of the Neron-Severi group of the very general projective K3 surface of some given degree), with the exception of cases when the genus is sufficiently close to the arithmetic genus: roughly the top $\nicefrac{1}{4}$ of the range for the primitive case \cite{[CD19]}, respectively roughly the top $\nicefrac{1}{6}$ of the range in general \cite{[Ke13]}. Unfortunately, the type of argument used to prove these results seems very difficult to use below the top $\nicefrac{1}{3}$ of the range asymptotically. Please see also \cite{[Ke05], [CD12], [Ba19], [De20], [CFGK17], [Ch99]}. 

K3 surfaces and abelian surfaces are similar in many ways, and this is also true of their Severi varieties. Although Severi varieties of K3 surfaces have been the subject of much research, the study of Severi varieties of abelian surfaces was initiated only quite recently by Knutsen, Lelli-Chiesa and Mongardi \cite{[KLM19]} in the analogous situation when the curve class is primitive (on the other hand, the answer to the curve counting problem in this setup has been long known, by work of Bryan and Leung \cite{[BL99]}). The authors propose the analogue of the conjecture discussed in the paragraph above: the Severi problem for general $(1,d)$-polarized abelian surfaces and curves in the primitive class \cite[Question 4.2]{[KLM19]}. 

The main goal of our paper is to answer this question. 

Let $(A,{\sh L})$ be a \emph{general} $(1,d)$-polarized abelian surface over $\cc$, and $\beta \in \mathrm{NS}(A)$ corresponding to $c_1({\sh L})$. First, we sketch our construction of the Severi variety. Note that if $f:C \to A$ is an unramified map from a smooth connected genus $g$ curve $C$ such that $f_*[C] = \beta$, then $f$ is birational onto its image, and in particular $2 \leq g \leq p_a(\beta) = d+1$ by the genus-degree formula. There is an open substack
$$ {\modu V}_g(A,\beta) \subset \overline{\modu M}_g(A,\beta) \quad \quad \text{(soon to be renamed $V_g(A,\beta)$)}$$
of the moduli stack of stable maps which consists of unramified stable maps with smooth sources. This follows from ${\modu M}_g \subset {\mathfrak M}_g$ open, \cite[\href{https://stacks.math.columbia.edu/tag/0475}{Tag 0475}]{[stacks]}, and properness of prestable curves. In fact, ${\modu V}_g(A,\beta)$ is (the stack associated to) a scheme, so in \S\ref{sec: intro} we will write $V_g(A,\beta)$ instead (however, our unsubstantial use of stacks will be a necessity of language throughout the main body of the paper). Indeed, ${\modu V}_g(A,\beta)$ is an algebraic space by \cite[Theorem 2.2.5, part (1)]{[Co07]} because it is a Deligne-Mumford stack whose objects have no nontrivial automorphisms, and hence a quasi-projective scheme since a projective coarse moduli space exists for $\overline{\modu M}_g(A,\beta)$ \cite[Theorem 1]{[FP95]}, and hence a quasi-projective one exists for ${\modu V}_g(A,\beta)$. Furthermore, Proposition \ref{prop: relatively smooth} implies that $V_g(A,\beta)$ is smooth of dimension $g$, hence it must be a disjoint union of smooth quasi-projective varieties of dimension $g$. The problem we wish to address is that of identifying these components. 

Recall that the $(1,d)$-polarization on $A$ induces a $(1,d)$-polarization on the dual abelian surface $A^\vee = \mathrm{Pic}^0(A)$, namely the isogeny $\smash{ \mu:A^\vee \to A }$ characterized by the property that $\mu \lambda$ is the multiplication by $d$ map, where $\lambda:A \to A^\vee$ is the given polarization on $A$. Let $\Sigma_g$ be the set of isogenies $\theta:A' \to A$ such that
$$ \displaystyle{ \deg \theta \in \left\{\frac{d}{\ell}: \ell \geq g-1 \text{ integer} \right\} \cap \zz } \quad \quad \text{ and } \quad \quad \text{$\mu$ factors through $\theta$,} $$
up to the obvious notion of isomorphism. 
For all $\theta \in \Sigma_g$, the subspace $V_g(A,\beta;\theta)$ of $V_g(A,\beta)$ consisting of maps $f:C \to A$ which factorize sharply through $\theta$, that is, $f$ factors through $\theta$ but not through any higher degree isogeny, is both open and closed, cf. \S\ref{ssec: factorization through isogenies} and Remark \ref{rmk: connect to intro}.

\begin{thm}\label{thm: main theorem}
Let $(A,{\sh L})$ be a general $(1,d)$-polarized abelian surface and $g$ such that $3 \leq g \leq d+1$. Then, with notation as above,
\begin{equation}\label{eqn:main theorem decomposition} V_g(A,\beta) = \bigsqcup_{\theta \in \Sigma_g} V_g(A,\beta;\theta) \end{equation}
is the decomposition of $V_g(A,\beta)$ into \textbf{irreducible} components.
\end{thm}

The omission of the case $g=2$ is due to the fact that it behaves differently, and in fact it is quite trivial: specifying a map from a smooth genus $2$ curve into an abelian surface is equivalent up to translation to specifying an isogeny into the abelian surface, by the Albanese universal property of Jacobians. For the subtle question of the singularities of (not necessarily primitive class) genus $2$ curves on abelian surfaces, please see \cite{[KL19]}.  

The number of components can be computed to be $ \smash{ |\Sigma_g| = \sum_{h|e|d, \mathop{} e \leq \frac{d}{g-1}} h }$. 
We return for a moment to the earlier discussion regarding the status of irreducibility depending on the position of $g$ in the range $[0,p_a(\beta)]$. In our situation, the Severi varieties are always irreducible for $g$ roughly in the top $\nicefrac{1}{2}$ of the range, but in the lower half they may well be reducible. As in the case of K3 surfaces, the result was previously known in roughly the top $\nicefrac{1}{4}$ of the range, by \cite[Theorem 1]{[De20]} and \cite[Propositions 3.2 and 3.3]{[BS91]} or \cite[Proposition 1.5]{[AM16]}.

Theorem \ref{thm: main theorem} has an `existence' part, that $V_g(A,\beta;\theta)$ is nonempty, and an `irreducibility' part, that $V_g(A,\beta;\theta)$ is irreducible. Both parts are proved by specializing to a product $E \times F$ of two elliptic curves, as in \cite{[BL99]}. (The most relevant feature of the product surface $E \times F$ is that all effective divisors $D \subset E \times F$ of algebraic class $\beta_{E \times F} = E \times \{[pt]\} + d \{[pt]\} \times F $ are sums of one fiber of $E \times F \to F$ and $d$ fibers of $E \times F \to E$.) The proof of the existence part is easy, but relies implicitly on nontrivial folklore ideas on the reduction of obstruction spaces by the semiregularity map. We will use the formulation in \cite{[KT14]}. The general principle to prove irreducibility was suggested in \cite{[Za19b]}: first degenerate to a highly reducible moduli space, then argue that its components coalesce when regenerating. However, the geometry is completely different from that in \cite{[Za19b]} and we don't rely on the theory of stable maps to degenerate targets \cite{[Li01]}. Instead, will use Theorem \ref{thm: necessary condition to deform}, which gives a fairly accurate criterion to determine which stable maps to $E \times F$ deform to nearby fibers, going well beyond the stable maps that come up in the calculation of the (reduced) Gromov-Witten invariants. Its proof is inspired by X. Chen's proof that primitive class rational curves on general K3 surfaces are nodal \cite{[Ch02]}. What Theorem \ref{thm: necessary condition to deform} does not provide in its current form is a fully satisfactory picture of how these `limit stable maps' vary in moduli. This can be done in ad hoc ways on sufficiently simple strata of the space of stable maps, and this is how we will proceed in this paper, although it would be highly desirable (and probably extremely useful) to have a more moduli-friendly version of this criterion, which would hopefully also apply to K3 surfaces; please see Question \ref{Q: main question for future}.

If, as in \cite{[KLM19]}, $V_{\{\sh L\},\delta}$ is the Severi variety of class $\beta$ curves $Y \subset A$ with $\delta = d+1-g$ nodes and no other singularities, then Theorem \ref{thm: main theorem} implies that two curves $Y_1,Y_2$ in $V_{\{\sh L\},\delta}$ belong to the same irreducible component of $\smash{ V_{\{\sh L\},\delta} }$ if and only if the maps from their normalizations $\smash{ \widetilde{Y}_i \to A }$ factor through the same maximal isogenies $A' \to A$ (using the `simultaneous resolution' of \cite[Th\'{e}or\`{e}me 1]{[Te80]} to construct the morphism $V_{\{\sh L\},\delta} \to V_g(A,\beta)$, which can even be shown to be an open immersion with a bit of work). What is not clear a priori is whether each $V_g(A,\beta;\theta)$ contains maps with nodal images. Fortunately, for $g \geq 5$, this follows from \cite[Theorem 1.3]{[KLM19]}, so, at least for $g \geq 5$, a decomposition completely analogous to \eqref{eqn:main theorem decomposition} holds for $V_{\{\sh L\},\delta}$ as well. The decomposition into irreducible components of the `universal' Severi variety can be deduced quite easily from the case of individual surfaces, and will not be written explicitly here. 

\subsection*{Acknowledgements} I would like to thank Brian Osserman for encouragement, and many useful discussions directly or indirectly related to this paper. In the final stages of writing this paper, I was supported by an NSERC Discovery Grant. 

\section{Families of Severi varieties}\label{sec: families of severi varieties}

\subsection{Factorization through isogenies}\label{ssec: factorization through isogenies} In this subsection, we will explain why the decomposition \eqref{eqn:main theorem decomposition} makes sense. Although factorizations through isogenies can be discussed simply in terms of fundamental groups, we adopt slightly more algebraic language, which handles degenerations and other future technicalities easier. In particular, we need to work in families in order to handle the specialization to $E \times F$ later in the paper. 

\begin{lem}\label{lem: pure isogeny calculation}
Let $V$ be a $(1,d)$-polarized abelian surface over an algebraically closed field of characteristic $0$, $\beta \in \mathrm{NS}(V)$ corresponding to the polarization, and $\mu:V^\vee \to V$ the induced $(1,d)$-polarization on the dual abelian surface. 

If $\theta:V' \to V$ is an isogeny of abelian surfaces, the following are equivalent:
\begin{enumerate}
\item there exists $\beta' \in \mathrm{NS}(V')$ such that $\theta_* \beta' = \beta$;
\item $\deg \theta$ divides $d$, and $\mu$ factors through $\theta$.
\end{enumerate}
Moreover, if these conditions are satisfied, $\beta'$ induces a polarization on $V'$ of type $(1,d')$, such that $d = d' \deg \theta$.
\end{lem}

\begin{proof} 
The first condition is equivalent to $\smash{ \theta^*\beta \in (\deg \theta) \mathrm{NS}(V') }$ because the composition $\smash{ \mathrm{NS}(V) \xrightarrow{\theta^*} \mathrm{NS}(V') \xrightarrow{\theta_*} \mathrm{NS}(V)}$ is $(\deg \theta) \mathrm{id}_{\mathrm{NS}(V)}$, and $\theta_* \otimes \mathrm{id}_{\mathbb Q}$ is an isomorphism. Thanks to the Lefschetz principle, we may safely assume that the ground field is $\cc$. Let us uniformize our abelian surfaces $V = \cc^2/\Lambda$ and $V' = \cc^2/ \Lambda'$, let $k:= \deg \theta = [\Lambda:\Lambda']$, and let the polarization on $V$ correspond to the type $(1,d)$ alternating bilinear form $\alpha : \Lambda \times \Lambda \to {\mathbb Z}$. Conditions (1) and (2) are equivalent to
\begin{enumerate} 
\item[$(1')$] $\alpha(\Lambda' \times \Lambda') \subseteq k {\mathbb Z}$; and respectively 
\item[$(2')$] $k|d$ and $\{x \in \Lambda: \alpha(x,y) \in d{\mathbb Z},\text{ for all } y \in \Lambda\} \subseteq \Lambda'$. 
\end{enumerate}
$(1) \Leftrightarrow (1')$ follows from the first sentence in this proof and the fact that the canonical identifications
$\smash{ \hh^2(V,\zz) \cong \mathrm{Hom}({\textstyle \bigwedge^2}\Lambda,\zz ) }$ and $\smash{ \hh^2(V',\zz) \cong \mathrm{Hom}({\textstyle \bigwedge^2}\Lambda',\zz ) }$ are compatible with pullbacks. $(2) \Leftrightarrow (2')$ is straightforward from the definition of $\mu$. 

The proof that $(1') \Leftrightarrow (2')$ is an elementary, yet not utterly trivial, exercise in lattice theory. Let $b_1,b_2,b_3,b_4$ be a symplectic basis of $\Lambda$, that is, a basis in which 
$\alpha$ is given by the block matrix $ \begin{bmatrix} 0 & D \\ -D & 0 \end{bmatrix}$, where $ D=\begin{bmatrix} 1 & 0 \\ 0 & d \end{bmatrix}$, 
and let $A = \langle b_1,b_3 \rangle$ and $B = \langle b_2,b_4 \rangle$. Note that $ \{x \in \Lambda: \alpha(x,y) \in d{\mathbb Z},\text{ for all } y \in \Lambda\} = dA \oplus B$. 
The implication $(2') \Rightarrow (1')$ is easy because $\Lambda' \supset B$ implies that $\Lambda' = A' \oplus B$ for some $A' \subseteq A$, and then everything clear. The key to $(1') \Rightarrow (2')$ is that the reverse inclusion of $(1')$ holds regardless. In general, if $\phi:L \times L \to \zz$ is an alternating bilinear form on a lattice $L$, and $L' \subset L$ is a sublattice of finite index, then 
\begin{equation}\label{eqn: lattices divisibility} [L:L'] \phi(L \times L) \subseteq \phi(L' \times L'), \end{equation}
and if equality occurs, then equality also occurs in the analogues of \eqref{eqn: lattices divisibility} for all lattices $L''$ such that $L' \subset L'' \subset L$.
This can be proved inductively using `Jordan-H\"{o}lder' filtrations $L' = L_0 \subset L_1 \subset \cdots \subset L_m = L$. Assuming that $(1')$ holds, the general observation above shows that there exists no $B' \subsetneq B$ such that $\Lambda' \subseteq A \oplus B'$. Hence $\Lambda' \supset B$, and again everything is clear. 

For the final claim, $\beta'$ is ample by the Nakai-Moishezon criterion and the push-pull formula, $\beta'$ is primitive since $\beta$ is primitive and $\theta_* \beta' = \beta$, and the type can be deduced by computing $\beta'^2 = 2d/\deg \theta$. (To shed some light on this rather opaque calculation, we mention that the polarization $\mu$ is the is the $\theta^\vee$-pullback of the polarization $\mu':V'^\vee \to V'$ on $V'^\vee$ induced by the one on $V'$. Then it follows from \cite[last formula on p. 143]{[Mu85]} that $\mu = \theta \mu' \theta^\vee$. Remark \ref{rmk: isogenies on product surfaces} should further demystify Lemma \ref{lem: pure isogeny calculation}.)
\end{proof}

Let $M$ be a smooth connected complex quasi-projective variety (in \S\ref{ssec: Severi stacks} we will make a more specific choice for $M$, but for now this is unimportant), and $V_M \to M$ an abelian scheme of relative dimension $2$ (a family of abelian surfaces), with a type $(1,d)$ polarization $\lambda_M:V_M \to V_M^\vee$ induced by an invertible sheaf ${\sh L}_M \in \mathrm{Pic}(V_M)$, and a canonical level structure given by $d^2$ sections $M \to V_M$, which thus give in each geometric fiber an identification of the kernel of $\lambda_M$ with $({\mathbb Z}/d)^2$. Let $\mu_M:V_M^\vee \to V_M$ be the $(1,d)$-polarization induced on the dual. As usual, we write $V_S = S \times_M V_M$, $\lambda_S = \mathrm{id}_S \times_M \lambda$, etc.

The set of all factorizations through isogenies $\smash{ V_M^\vee \to V' \xrightarrow{\theta} V_M }$ of $\mu_M$ up to the obvious notion of isomorphism is canonically identified with the set of subgroups of $({\mathbb Z}/d)^2$, and naturally ordered by reverse inclusion. For instance, $\smash{ V_M^\vee = V_M^\vee \to V_M }$ is the maximal factorization. Let $g>1$. It is convenient to consider the functors
$$ \underline{\Sigma}:\mathrm{Sch}^\mathrm{op}_M \longrightarrow \mathrm{Set} \quad \text{and} \quad  \underline{\Sigma}_g:\mathrm{Sch}^\mathrm{op}_M \longrightarrow \mathrm{Set} $$
which respectively associate to any $S \to M$ the sets of factorizations $\smash{ V_S^\vee \to V' \xrightarrow{\theta} V_S }$ of $\mu_S$ such that $\deg \theta$ divides $d$ (at all points of $S$), respectively $\deg \theta$ divides $d$ and $\deg \theta \leq \frac{d}{g-1}$. By abuse of language, we will often refer to $\smash{ V_S^\vee \to V' \xrightarrow{\theta} V_S }$ as simply $\theta$. Let $\smash{ \Sigma = \underline{\Sigma}(M) }$ and $\smash{\Sigma_g = \underline{\Sigma}_g(M) }$. 

\begin{rmk}\label{rmk: connect to intro isogenies}
Since $V_M$ has a canonical level structure and $M$ is connected, we have natural identifications $\smash{\underline{\Sigma}(\overline{s}) = \underline{\Sigma}(M) }$ and $\smash{\underline{\Sigma}_g(\overline{s}) = \underline{\Sigma}_g(M) }$, for all geometric points $\overline{s} \to M$. In particular, the definition $\smash{\Sigma_g = \underline{\Sigma}_g(M) }$ is consistent with the one used in the introduction.
\end{rmk}
 
Fix $3 \leq g \leq d+1$, and let
$\smash{ \overline{\modu M}_{g,n}(V_M/M,\beta) }$
be the relative moduli stack of genus $g$, $n$-marked stable maps of primitive class.  A stable map $(C,f,x_1,\ldots,x_n)$ over $S \to M$ is `of primitive class' if, for any geometric point $\overline{s} \to S$ and any ${\sh J} \in \mathrm{Pic}(V_{\overline{s}})$, we have $\deg f_{\overline{s}}^* {\sh J} = ({\sh J} \cdot {\sh L}_{\overline{s}})_{V_{\overline{s}}}$. We write $\smash{ \overline{\modu M}_{g,n}(V_M/M,\beta)(S) }$ for the groupoid of stable maps over $S \to M$. By Lemma \ref{lem: pure isogeny calculation}, a similar discussion holds on $V'$, for any $\smash{ (V_M^\vee \to V' \xrightarrow{\theta} V_M) \in \Sigma }$. We denote the primitive class on $V'$ by $\beta' = \beta'(\theta)$, and we thus have $\theta_*\beta' = \beta$. 

\begin{lem}\label{lem: possible factorizations}
Let $V = V_{\overline{s}}$ be a geometric fiber of $V_M \to M$, and $(C,f,x_1,\ldots,x_n)$ a stable map in $\overline{\modu M}_{g,n}(V,\beta_{\overline{s}})(\overline{s})$. 
\begin{enumerate}
\item If $f$ factors through an isogeny $\theta: V' \to V$, then $\smash{ \theta \in \underline{\Sigma}(\overline{s}) }$. If, in addition, $C$ is irreducible, then $\smash{ \theta \in \underline{\Sigma}_g(\overline{s}) }$.
\item The subset $\smash{ \{\theta \in \underline{\Sigma}(\overline{s}) = \Sigma: \text{ $f$ factors through $\theta$} \} }$ of $\Sigma$ has a maximum. 
\end{enumerate}
We call the maximum above the \emph{maximal factorization of $f$}. 
\end{lem}

\begin{proof}
Let $f':C \to V'$ such that $f = \theta f'$. Since $\theta_* f'_*[C] = f_*[C]$, the first part of (1) follows from Lemma \ref{lem: pure isogeny calculation}. If $C$ is irreducible, then $f$ must be birational onto its image because the polarization is primitive, and hence so must $f'$. Then
$$ g = p_a(C) \leq p_a(\beta') = d'+1 = \frac{d}{\deg \theta} + 1, \quad \text{so} \quad \deg \theta \leq \frac{d}{g-1} $$
by the genus-degree formula and Lemma \ref{lem: pure isogeny calculation}. If $f$ factors through isogenies $V' \to V$ and $V'' \to V$, then it also factors through (the restriction to a connected component of the source of) $V' \times_V V'' \to V$, which must also be an element of $\smash{ \underline{\Sigma}(\overline{s}) }$ by part (1), so the maximum in part (2) exists. 
\end{proof}

We may now state the main result of \S\ref{ssec: factorization through isogenies}.

\begin{prop}\label{prop: total decomposition}
There exists a decomposition into open and closed substacks
\begin{equation}\label{eqn: total decomposition}
 \overline{\modu M}_{g,n}(V_M/M,\beta) = \bigsqcup_{\theta \in \Sigma}  \overline{\modu M}_{g,n}(V_M/M,\beta;\theta),
\end{equation} 
such that a stable map $(C,f,x_1,\ldots,x_n)$ in $\overline{\modu M}_{g,n}(V_M/M,\beta)(S)$ is an object of $\overline{\modu M}_{g,n}(V_M/M,\beta;\theta)(S)$ if and only if, for any geometric point $\overline{s} \to S$, the maximal factorization of $f_{\overline{s}}$ in the sense of Lemma \ref{lem: possible factorizations} is $\theta \in \underline{\Sigma}(\overline{s})$. In this case, $f$ also factors through $\smash{ \theta \in \underline{\Sigma}(S) \supseteq \Sigma }$. 
\end{prop}

(If $S$ is of finite type over $\cc$, then it obviously suffices to impose the condition above only on $\cc$-points.)

\begin{proof}
Let $\theta \in \Sigma$. Note that the composition of any stable map to $V'$ with $\theta$ is a stable map into $V_M$, so we obtain a $1$-morphism of Deligne-Mumford stacks 
$$ \tau = \tau(\theta):\overline{\modu M}_{g,n}(V'/M,\beta') \longrightarrow \overline{\modu M}_{g,n}(V_M/M,\beta), $$
where $\beta' = \beta'(\theta)$. We claim that $\tau$ is proper and \'{e}tale. It suffices to check that it is formally \'{e}tale -- everything else is clear, as both the source and the target are proper over $M$. Let $j:T \hookrightarrow S$ be a closed immersion of $M$-schemes induced by a nilpotent sheaf of ideals on $S$, and consider a (solid arrow) commutative diagram
\begin{center}
\begin{tikzpicture}
\matrix [column sep  = 10mm, row sep = 5mm] {
	\node (nw) {$T$}; &
	\node (ne) {$\overline{\modu M}_{g,n}(V'/M,\beta')$};  \\
	\node (sw) {$S$}; &
	\node (se) {$\overline{\modu M}_{g,n}(V_M/M,\beta)$.}; \\
};
\draw[->, thin] (nw) -- (ne);
\draw[->, thin] (ne) -- (se);
\draw[right hook->, thin] (nw) -- (sw);
\draw[->, thin] (sw) -- (se);
\draw[->, dashed, thin] (sw) -- (ne);
\node at (1,0) {$\tau$};
\end{tikzpicture}
\end{center}
We have to show that there exists a unique dotted arrow which makes the diagram commute. We have stable maps $f':C_T \to V'_T$ and $f:C \to V_S$ over $T$ and $S$ respectively, with markings $x_i:S \to C$ and $x_i':T \to C_T$, $i=1,\ldots,n$, such that 
$$ (\theta \times_M j) \circ f'= f|{C_T} \quad \text{and} \quad x_i|T=x'_i. $$ 
Then our task is to prove that there exists a unique dotted arrow such that
\begin{center}
\begin{tikzpicture}
\matrix [column sep  = 10mm, row sep = 5mm] {
	\node (n) {$V'_T$}; &
	\node (ne) {$V'_S$}; &
	\node (se) {$V_S$}; \\
	\node (nw) {$C_T$}; & &
	\node (sw) {$C$}; \\
};

\draw[->, thin] (nw) -- (n);
\draw[right hook->, thin] (n) -- (ne);
\draw[->, thin] (ne) -- (se);
\draw[right hook->, thin] (nw) -- (sw);
\draw[->, thin] (sw) -- (se);
\draw[->, dashed, thin] (sw) -- (ne);
\node at (0.95,0.75) {$\theta_S$};
\node at (1.9,0) {$f$};
\node at (-1.95,0) {$f'$};
\end{tikzpicture}
\end{center}
commutes (the part involving the markings is a tautology thus skipped), which follows from the fact that $\theta_S: V'_S \to V_S$ is formally \'{e}tale. We explain the reformulation of the task in more detail. Trivially, the bottom right triangle in the former diagram commutes if and only if the top right triangle in the latter diagram commutes. Slightly less trivially, the top left triangle in the former diagram commutes if and only if the square (trapezoid) in the latter diagram commutes. A priori, the equivalence is with the square being cartesian; however, if the square commutes, then it is a fortiori cartesian. Indeed, the square induces a morphism $C_T \to V'_T \times_{V'_S} C \cong C_T$ compatible with $C_T \hookrightarrow C$, but there are no nontrivial endomorphisms of $C_T$ over $C$, so $C_T \to C_T$ must have been the identity.

It follows that the map $|\overline{\modu M}_{g,n}(V'/M,\beta')| \to |\overline{\modu M}_{g,n}(V_M/M,\beta)|$ induced by $\tau$ on the underlying topological spaces is both open and closed. (It is well-known that \'{e}tale morphisms of schemes are open \cite[\href{https://stacks.math.columbia.edu/tag/03WT}{Tag 03WT}]{[stacks]}; the reductions needed to use this for our Deligne-Mumford stacks require only the definition of the associated topological space \cite[\href{https://stacks.math.columbia.edu/tag/04XI}{Tag 04XI}]{[stacks]}.) In particular, the image of this map is open and closed, thus corresponds to an open and closed substack of $\overline{\modu M}_{g,n}(V_M/M,\beta)$, which we denote $\smash{{\modu F}_{\theta}}$. Trivially, $\smash{ |{\modu F}_{\theta'}| \subseteq |{\modu F}_{\theta}| }$ if $\theta' > \theta$. There exists a unique open and closed substack $\overline{\modu M}_{g,n}(V_M/M,\beta;\theta)$ of $\overline{\modu M}_{g,n}(V_M/M,\beta)$ such that 
$$  |\overline{\modu M}_{g,n}(V_M/M,\beta;\theta)| = |{\modu F}_{\theta}| \backslash \bigcup_{\theta'>\theta} |{\modu F}_{\theta'}|, $$
or equivalently, $(C,f,x_1,\ldots,x_n)$ is in $\overline{\modu M}_{g,n}(V_M/M,\beta;\theta)(S)$ if and only if $f$ factors through $\theta$, seen as an element of $\smash{\underline{\Sigma}(S) }$, and $f$ does not factor through $\theta'$, for any $\theta' > \theta$, and neither does any pullback of $f$ along any morphism $S' \to S$. The former description and Lemma \ref{lem: possible factorizations} prove the defining property of $\overline{\modu M}_{g,n}(V_M/M,\beta;\theta)$ required in the statement of the proposition, while the latter description trivially implies the latter claim. 
\end{proof}

\begin{rmk}\label{rmk: compatible with forgetful} The case $n>0$ will be used only once, in the proof of Proposition \ref{prop: flexibility}. We will implicitly use the fact that the decompositions \eqref{eqn: total decomposition} are compatible with the $1$-morphisms $\overline{\modu M}_{g,n}(V_M/M,\beta) \to \overline{\modu M}_{g,m}(V_M/M,\beta)$ which forget some or all of the marked points, $n > m$. (Recall that these forgetful morphisms can be constructed based on \cite[Corollary 4.6]{[BM96]}, and no substantial modifications are required in a relative setting.) \end{rmk}

\subsubsection{Factorizations through isogenies of elliptic curves}\label{sss: factorizations through isogenies of elliptic curves} An analogous, but much less involved discussion holds if we replace the abelian surfaces with elliptic curves. Let $E$ be a complex projective smooth genus $1$ curve, and $\smash{ \overline{\modu M}_{g,n}(E,d) }$ the moduli space of $n$-marked degree $d$ genus $g$ stable maps. By arguments similar to those in the proof of Proposition \ref{prop: total decomposition}, we have a decomposition
\begin{equation}\label{eqn: total decomposition elliptic curve}
 \overline{\modu M}_{g,n}(E,d) = \bigsqcup_{\deg \theta | d} \overline{\modu M}_{g,n}(E,d;\theta) 
\end{equation}
over isogenies $\theta:E' \to E$ such that $\deg \theta | d$, such that a stable map in $\smash{\overline{\modu M}_{g,n}(E,d)}$ belongs to $\overline{\modu M}_{g,n}(E,d;\theta)$ if and only if all its geometric fibers factor through $\theta_{\overline{K}}$, but not through any higher degree isogeny. Note that isogenies $\theta:E' \to E$ such that $\deg \theta | d$ correspond to sublattices $\Lambda \subseteq \mathrm{H}_1(E,\zz)$ such that $[\mathrm{H}_1(E,\zz): \Lambda] | d$, and we will often write $\smash{\overline{\modu M}_{g,n}(E,d;\Lambda) }$ instead of $\smash{\overline{\modu M}_{g,n}(E,d;\theta) }$. The analogue of Remark \ref{rmk: compatible with forgetful} regarding forgetting the markings still applies. 

\subsection{The moduli stacks}\label{ssec: Severi stacks} We begin by making a more convenient choice for the base $M$ from \S\ref{ssec: factorization through isogenies}. Let ${\modu A}^{\mathrm{lev}}_{1,d}$ be moduli space of $(1,d)$-polarized abelian surfaces with a canonical level structure. Let $M$ be a smooth, connected, quasi-projective threefold, with a (smooth, nonempty) divisor $\partial M \subset M$, such that the $1$-morphism
$ M \to {\modu A}^{\mathrm{lev}}_{1,d} $
induced by $V_M \to M$ is \'{e}tale, and moreover
\begin{enumerate} 
\item the points $p \in M\backslash \partial M(\cc)$ correspond to abelian surfaces $V_p$ for which the polarization is not the sum of two nonzero effective classes in ${\mathrm{NS}}(V_p)$.
\item the points $p \in \partial M(\cc)$ correspond to products of elliptic curves, that is, $V_p = E \times F$ where $E$ and $F$ are elliptic curves, and $\lambda_p = (d\mathrm{id}_E,\mathrm{id}_F)$ or equivalently $\beta_p = E \times \{[pt]\} + d \{[pt]\} \times F \in \hh_2(E \times F,\zz)$. 
\end{enumerate} 
The fact that such $M$ and $\partial M$ exist follows from standard abelian surface theory, namely, from the fact that ${\modu A}^{\mathrm{lev}}_{1,d}$ is a smooth, separated, finite type Deligne-Mumford stack over $\cc$ of dimension $3$, and from standard facts on Noether-Lefschetz loci \cite{[Vo13]} of abelian surfaces, which also follow from the complex analytic theory. 

\begin{rmk}\label{rmk: isogenies on product surfaces}
There is a very concrete description of $\smash{\underline{\Sigma}(p)}$ and $\smash{\underline{\Sigma}_g(p)}$ at points $p \in \partial M(\cc)$. Let $V_p = E \times F$, such that
$\lambda_p = (d\mathrm{id}_E,\mathrm{id}_F)$ and $\mu_p = (\mathrm{id}_E,d\mathrm{id}_F)$. Then $\smash{\underline{\Sigma}(p)}$ (respectively $\smash{\underline{\Sigma}_g(p)}$) is the set of isogenies of the form $\theta =(\mathrm{id}_E,\theta_F) : E \times F' \to E \times F$ such that $\deg \theta_F | d$ (respectively $\deg \theta_F | d$ and $\smash{ \deg \theta_F \leq \frac{d}{g-1} }$), and is naturally identified with the set of sublattices of $\mathrm{H}_1(F,\zz) \approx \zz^2$ whose index divides $d$ (respectively divides $d$ and doesn't exceed $\smash{ \frac{d}{g-1} }$). Indeed, any sublattice whose index divides $d$ necessarily contains $d\mathrm{H}_1(F,\zz)$. 
\end{rmk}

We now introduce two open substacks of $\overline{\modu M}_{g,n}(V_M/M,\beta)$. Let ${\modu T}_{g,n}(V_M/M,\beta)$ be the open substack of $\overline{\modu M}_{g,n}(V_M/M,\beta)$ consisting of stable maps with sources of compact type. Let ${\modu V}_{g,n}(V_M/M,\beta)$ be the open substack of ${\modu T}_{g,n}(V_M/M,\beta)$ consisting of unramified stable maps with sources of compact type. As customary, when $n=0$, we suppress this index. By \eqref{eqn: total decomposition}, we have decompositions
\begin{equation}\label{eqn: total decomposition X}
{\modu X}_{g,n}(V_M/M,\beta) = \bigsqcup_{\theta \in \Sigma}  {\modu X}_{g,n}(V_M/M,\beta;\theta), \quad \text{for all symbols ${\modu X}={\modu T},{\modu V},\overline{\modu M}$.}
\end{equation} 
We will see later \eqref{eqn: total decomposition V} that for ${\modu X} = {\modu V}$, $\Sigma$ can (and should) be replaced with $\Sigma_g$. 

\begin{rmk}\label{rmk: connect to intro}
All stable maps in ${\modu V}_g(V_M/M,\beta)_{M \backslash \partial M}$ have smooth sources and (their geometric fibers) are birational onto their images. In particular, the fibers of the moduli spaces ${\modu V}_g(V_M/M,\beta)$ and ${\modu V}_g(V_M/M,\beta;\theta)$ over points $p \in M \backslash \partial M({\mathbb C})$ are precisely the spaces in Theorem \ref{thm: main theorem}, by Remark \ref{rmk: connect to intro isogenies}.
\end{rmk}

Throughout much of the paper, we will be working in the following situation.

\begin{sit}\label{sit: main situation}
Fix $3 \leq g \leq d+1$, and $\theta \in \Sigma_g$. Let $\tau:\Delta \to M$ be a morphism from a nonsingular connected quasi-projective curve $\Delta$, such that $\tau^{-1}(\partial M)$ is \emph{set-theoretically} a single point $o \in \Delta(\cc)$. Let $\Delta^* = \Delta \backslash \{o\}$.
\end{sit}

In Situation \ref{sit: main situation}, let $\Delta^* = \Delta \backslash \{o\}$, $V = \Delta \times_M V_M$, ${\sh L}$ the pullback of ${\sh L}_M$ to $V$, and $V_o = E \times F$, with ${\sh L}_o = {\sh J}_E \boxtimes {\sh J}_F$ with $\deg {\sh J}_E = d$ and $\deg {\sh J}_F = 1$. Let 
$$ {\modu X} = \Delta \times_M {\modu X}_g(V_M/M,\beta;\theta) \quad \text{and} \quad {\modu X}^*=\Delta^* \times_\Delta {\modu X}, \quad \text{for all ${\modu X}={\modu T},{\modu V},\overline{\modu M}$.} $$
By Remarks \ref{rmk: connect to intro isogenies} and \ref{rmk: isogenies on product surfaces}, $\theta \in \Sigma_g$ corresponds to a sublattice $\Lambda[\theta] \subset \mathrm{H}_1(F,\zz)$. 

\begin{prop}\label{prop: relatively smooth}
The projection ${\modu V} \to \Delta$ is smooth of relative dimension $g$. In particular, ${\modu V}$ is smooth of dimension $g+1$. 
\end{prop}

\begin{proof}
We first review some generalities. Let $f:C \to X$ be an unramified stable map (over $\spec \cc$) to a smooth projective surface $X$. We review the fact that the space of first order deformations of $f$ is $\mathrm{Hom}(f^*\omega_X,\omega_C)$. In general, that is, without the `unramified' hypothesis, this space is $\smash{ \mathbb{E}\mathrm{xt}^0(f^*\Omega_X \to \Omega_C,{\sh O}_C) }$ with $f^*\Omega_X$ in degree $0$ \cite{[Be97], [BF97]}. However, if $f$ is unramified, it simplifies to $\mathrm{Hom}(f^*\omega_X,\omega_C)$, in light of the short exact sequence
\begin{equation}\label{eqn: mysterious ses} 
0 \longrightarrow {\sh H}\!om(\omega_C,f^*\omega_X) \longrightarrow f^*\Omega_X \longrightarrow \Omega_C \longrightarrow 0. 
\end{equation}
To justify \eqref{eqn: mysterious ses}, we first note that $f^*\Omega_X \to \Omega_C$ is surjective with invertible kernel. Indeed, \cite[\href{https://stacks.math.columbia.edu/tag/04HI}{Tag 04HI}]{[stacks]} reduces this to the obvious fact that $\{xy=0\} \hookrightarrow {\mathbb A}^2$ is lci of codimension $1$. However, if ${\sh K}$ is the kernel of the surjective map $f^*\Omega_X \to \Omega_C$, then $\bigwedge^2 f^*\Omega_X \otimes {\sh K}^\vee \cong \omega_C$ by one of the definitions of the dualizing sheaf \cite[p. 163]{[Kn83]}, and \eqref{eqn: mysterious ses} follows.

Returning to the proof of the proposition, let $f:C \to V_t$ be a stable map in ${\modu V}_t(\mathbb C)$, $t \in \Delta(\cc)$. The discussion above shows the space of first order deformations of $f$ as a stable map to $V_t$ is $g$-dimensional. Thus ${\modu V}/\Delta$ has relative dimension at most $g$ and it is smooth if equality occurs. In fact, it follows by well-known ideas related to the `semi-regularity map' that what we've proved so far suffices to deduce the proposition. Indeed, by \cite[Theorem 2.4 and Remark 3.1]{[KT14]}, $\smash{\overline{\modu M} \to \Delta}$ admits a relative perfect obstruction theory $\smash{ E_\mathrm{red}^\bullet \to {\mathbb L}_{\overline{\modu M}/\Delta} }$ of relative dimension $g$. The transversality condition \cite[condition (3)]{[KT14]} is trivial to check \cite[Remark 1, p. 8]{[Mu85]}. Hence $\smash{{\modu V} \to \Delta}$ also admits a relative perfect obstruction theory of dimension $g$ since ${\modu V}$ is open in $\smash{\overline{\modu M}}$. However, we've just shown that the spaces of (relative) first order deformations at all $\cc$-points of ${\modu V}$ are $g$-dimensional, hence all \emph{reduced} obstruction spaces vanish, and the conclusion follows \cite[Proposition 7.3]{[BF97]}. 
\end{proof}

In particular, Proposition \ref{prop: relatively smooth} shows that there is no distinction between connected and irreducible components of ${\modu V}$, by \cite[Proposition 4.16]{[DM69]}. 

We also note that ${\modu V}_g(V_M/M,\beta;\theta) = \emptyset$ if $\theta \notin \Sigma_g$, and so
\begin{equation}\label{eqn: total decomposition V}
{\modu V}_{g,n}(V_M/M,\beta) = \bigsqcup_{\theta \in \Sigma_g}  {\modu V}_{g,n}(V_M/M,\beta;\theta).
\end{equation} 
Indeed, for $p \in M \backslash \partial M(\cc)$, we have ${\modu V}_g(V_M/M,\beta;\theta)_p = \emptyset$ if $\theta \notin \Sigma_g$ by Lemma \ref{lem: possible factorizations} and Remark \ref{rmk: connect to intro}, and (although this is somewhat of a cheat) the claim can then be deduced from Proposition \ref{prop: relatively smooth}. Alternatively, a direct elementary argument that ${\modu V}_g(V_M/M,\beta;\theta)_p = \emptyset$ for $\theta \notin \Sigma_g$, $p \in \partial M(\cc)$ is also possible, and involves analyzing stable maps similar to those in Definition \ref{defn: simple maps} below, but we are free to skip this. 

It will be convenient to have an explicit way to compute the maximal factorization (Lemma \ref{lem: possible factorizations}) of stable maps in the central fiber. We will only need this in very concrete situations, so we state a trivial criterion which will suffice for our needs.

\begin{rmk}\label{rmk: factorizations on product surfaces}
Let $f:C \to V_o$ be a genus $g$ primitive class stable map (over the base $\spec \cc$, i.e. a `single' stable map), and assume that $C$ is of compact type and decomposes into irreducible components as $\smash{ C = B \cup \bigcup_{k=1}^r G_k \cup \bigcup_{\alpha=1}^\ell T_\alpha }$, such that $B \cong E$ is the root of the dual tree, and maps isomorphically onto a fiber of $V_o \to F$, $G_k$ is a ghost (contracted) rational component, and $T_\alpha$ is a positive genus (say, non-contracted) component mapping onto a fiber of $V_o \to E$, and $T_\alpha$ corresponds to a leaf of the dual tree. Then we may trivially compute the maximal factorization of $f$ as follows: if $\Lambda_\alpha \subseteq \mathrm{H}_1(F,\zz)$ is the image of the map on $\hh_1$ induced by $T_\alpha \to F$, then the maximal factorization of $f$ is the isogeny that corresponds to $\smash{ \sum_{\alpha = 1}^\ell \Lambda_\alpha }$ cf. Remark \ref{rmk: isogenies on product surfaces}. Indeed, $f$ factors through some $\theta =(\mathrm{id}_E,\theta_F) : E \times F' \to E \times F$ if and only if all maps $T_\alpha \to F$ factor through $\theta_F$, which happens if and only if $\Lambda_\alpha$ is contained in the image of $\hh_1(\theta_F)$. 
\end{rmk}

\subsection{Combinatorial preliminaries}\label{ssec: partitions combinatorics} This subsection is concerned with some elementary combinatorial considerations. Although nothing here is too deep or difficult, the motivation for our definitions will only gradually become clear. Let $\zz^2$ be a lattice of rank two. 

\begin{defn}\label{defn: partitions}
A \emph{partition} of the lattice $\zz^2$ is an unordered collection $\Lambda_1,\ldots,\Lambda_k$ of rank $2$ sublattices that collectively span $\zz^2$, that is, $\zz^2 = \Lambda_1 + \cdots + \Lambda_k$.  
The \emph{length} of the partition is $k$, the number of summands, and the \emph{degree} of the partition $\varpi$ is $\deg \varpi = [\zz^2:\Lambda_1] + \cdots + [\zz^2:\Lambda_k]$. 
\end{defn}

Let $\Pi^k$ be the set of partitions of $\zz^2$ of length $k$, and $\Pi_d^k$ the set of partitions of $\zz^2$ of degree $d$ and length $k$.

\begin{rmk}\label{rmk: partitions exist}
If $2 \leq k \leq d$, then $\Pi_d^k \neq \emptyset$.
\end{rmk}

We draw an arrow $\varpi_1 \succ \varpi_2$ pointing from $\varpi_1 \in \Pi^k$ to $\varpi_2 \in \Pi^{k-1}$ if $\varpi_2$ is obtained from $\varpi_1$ by replacing two summands $\Lambda_i$ and $\Lambda_j$ of $\varpi_1$ with their sum $\Lambda_i + \Lambda_j$. We define an unoriented graph structure on $\Pi^k$ by drawing an edge between two elements $\varpi_1,\varpi_2 \in \Pi^k$ if:
\begin{enumerate}
\item there exists some $\varpi_3 \in \Pi^{k-1}$ with arrows $\varpi_1 \succ \varpi_3$ and $\varpi_2 \succ \varpi_3$, and
\item the new term of $\varpi_3$ coming from $\varpi_1 \succ \varpi_3$ is also the new term of $\varpi_3$ coming from $\varpi_2 \succ \varpi_3$.
\end{enumerate} 
Under these circumstances, we say that $\varpi_1 \succ \varpi_3 \prec \varpi_2$ is a \emph{roof}. We take the induced graph structure on $\Pi_d^k$. Equivalently, the graph structure on $\Pi_d^k$ can be described as follows: $\varpi=(\Lambda_1,\ldots, \Lambda_k)$ and $\varpi'=(\Lambda_1,\ldots, \Lambda_k)$ are connected by an edge if and only if there exist $i<j$ and $i'<j'$ such that
$$ \Lambda_i+\Lambda_j = \Lambda'_{i'}+\Lambda'_{j'} \quad \text{and} \quad [\zz^2:\Lambda_i] + [\zz^2:\Lambda_j] = [\zz^2:\Lambda'_{i'}] + [\zz^2:\Lambda'_{j'}]. $$ 

\begin{prop}\label{prop: graph is connected}
The graph structure on $\Pi_d^k$ above is connected if $2 \leq k \leq d$. 
\end{prop}

\begin{proof}
We proceed by induction on $k$. The graph is complete for $k=2$. Assume that $k \geq 3$ and that the statement holds for all $(d',k')$ with $k'<k$. Let $\varpi \in \Pi_d^k$. We say that $\varpi$ is equivalent to $\varpi'$ if it belongs to the same connected component.

The first step is to prove that $\varpi$ is equivalent to a partition containing $\zz^2$ among its terms. Note that any pair of summands $(\Lambda_i,\Lambda_j)$ can be replaced with any pair $(\Lambda_i+\Lambda_j,\Lambda')$, where $\Lambda' \subset \Lambda_i+\Lambda_j$ has the suitable index $ [\Lambda_i+\Lambda_j: \Lambda'] = [\Lambda_i+\Lambda_j: \Lambda_i] + [\Lambda_i+\Lambda_j: \Lambda_j]-1$. 
The only way the value of the smallest index cannot be decreased is if a sublattice realizing this minimum contains all the other summands. This lattice must be $\zz^2$.

The second step is to prove that  $\varpi$ is equivalent to a partition containing $\smash{ \zz^2 }$ such that the remaining lattices also add up to $\smash{ \zz^2 }$. This will complete the proof by the inductive hypothesis since we can play the game only with the remaining $k-1$ lattices. Replace $\varpi$ with the partition found in the first step. Now any lattice can be replaced with any other lattice of the same index. It is not hard to see that given any sequence $i_1,\ldots,i_{k-1}$ of natural numbers, there exist $k-1$ lattices of these indices which add up to $\zz^2$. For instance, we can choose the first lattice to be $\zz \times i_1 \zz$, the second lattice to be  $i_2\zz \times \zz$ and the others arbitrarily. \end{proof}

\subsection{Reduction of Theorem \ref{thm: main theorem} to Claim \ref{claim: key claim}}\label{ss: reduction ss} Let $\Pi = \Pi_d^{g-1}$ be the set of partitions of $\Lambda[\theta]$ of length $g-1$ and degree $d$, cf. Definition \ref{defn: partitions}, and with notation as in \S\ref{ssec: partitions combinatorics} and before Proposition \ref{prop: relatively smooth}. Keeping Remark \ref{rmk: factorizations on product surfaces} in mind, we make the following definition. 

\begin{defn}\label{defn: simple maps}
Let $\varpi = (\Lambda_1 + \cdots + \Lambda_{g-1} = \Lambda[\theta]) \in \Pi$. A stable map $(C,f) \in {\modu T}_o(\cc)$ is \emph{simple of type $\varpi$} if
\begin{enumerate}
\item $C = B \cup T_1 \cup \cdots \cup T_{g-1}$ such that all irreducible components $B,T_1,\ldots,T_{g-1}$ are smooth of genus $1$, and the dual graph is a star centered at $B$; and
\item the restriction of $f$ to the `backbone' $B$ is an isomorphism onto a fiber of $V_o \to F$, and the restriction of $f$ to the `tooth' $T_i$ is the isogeny onto a fiber of $V_o \to E$ corresponding to $\Lambda_i$. 
\end{enumerate}
\end{defn}

These are essentially the stable maps that come up in \cite[\S4]{[BL99]}, without the marked points on the source.

\begin{rmk}\label{rmk: simple maps are unramified}
All simple stable maps of type $\omega$ belong to ${\modu V}(\cc)$. 
\end{rmk}

Definition \ref{defn: simple maps} trivially extends to any algebraically closed extension of $\cc$, and then it extends to families by requiring the geometric fibers to be simple of type $\varpi$. The category of simple stable maps of type $\varpi$ is an open substack of ${\modu V}_o$, which we will denote by ${\modu V}[\varpi]$. We sketch the construction and leave the details to the reader. We have a `clutching' $1$-morphism 
\begin{equation}\label{eqn: clutching for simple} \left\{(x_1,\ldots,x_{g-1}) \in E^{g-1}:x_i \neq x_j \text{ for all $i \neq j$}\right\}  \times F \longrightarrow {\modu V}_o, \end{equation}
which, on $\cc$-points, glues a fiber of $V_o \to F$ to $g-1$ isogenies corresponding to $\Lambda_1,\ldots,\Lambda_{g-1}$ onto distinct fibers of $V_o \to E$ in the natural way required to obtain a type $\varpi$ simple stable map. This can be formally defined simply by constructing the corresponding stable map over its source. It can be checked that \eqref{eqn: clutching for simple} is \'{e}tale (some technical details are left to the reader here), and we take ${\modu V}[\varpi]$ to be the open substack corresponding to its image, keeping \cite[\href{https://stacks.math.columbia.edu/tag/04XI}{Tag 04XI}]{[stacks]} in mind. In particular, ${\modu V}[\varpi]$ is nonempty and irreducible of dimension $g$.

\begin{rmk}\label{rmk: remark about closures} Let ${\modu Z}$ be an irreducible, or equivalently, connected, component of ${\modu V}$, and $\smash{ \overline{{\modu Z}}}$ its closure relative to either ${\modu T}$ or $\smash{\overline{\modu M}}$ with, say, the reduced structure. We claim that if $\smash{ |\overline{{\modu Z}}| \supset |{\modu V}[\varpi]|}$ for some $\varpi \in \Pi$, then $\smash{ |{\modu Z}| \supset |{\modu V}[\varpi]|}$. (This is elementary, but rather confusing because $\overline{{\modu Z}_o} \neq \overline{\modu Z}_o$, so we write a careful proof.) Let $v \in |{\modu V}[\varpi]|$, arbitrary. Let ${\modu Z}'$ be the connected component of ${\modu V}$ which contains $v$, so $|{\modu Z}'| \cap |\overline{\modu Z}| \neq \emptyset$. Since $|\overline{\modu Z}|$ is irreducible, and $|{\modu Z}'| \cap |\overline{\modu Z}|$ and $|{\modu Z}|$ are both nonempty and open in it, we must have $\emptyset \neq (|{\modu Z}'| \cap |\overline{\modu Z}|) \cap |{\modu Z}| = |{\modu Z}| \cap |{\modu Z}'|$. Hence ${\modu Z} = {\modu Z}'$ and $v \in |{\modu Z}|$.
\end{rmk}

\begin{prop}\label{prop: flexibility}
Any irreducible component of ${\modu V}$ contains some ${\modu V}[\varpi]$, $\varpi \in \Pi$. 
\end{prop}

\begin{proof}
First, we claim that on the source $C$ of any stable map $(C,f) \in {\modu V}^*({\cc})$, we may add $g$ markings ${\mathbf x}=(x_1,\ldots,x_g)$ such that the $\Delta$-relative differential of the evaluation $1$-morphism
\begin{equation} \mathbf{ev}: \overline{\modu M}_{g,g}(V/\Delta,\beta) \longrightarrow V_\Delta^g := \underbrace{V \times_\Delta \cdots \times_\Delta V}_{\text{$g$ copies of $V$}} \end{equation}
is an isomorphism at the point $(C,{\mathbf x},f)$. By Remark \ref{rmk: connect to intro}, $C$ is smooth of genus $g$. Let $D= x_1 + \cdots + x_g$ be a general degree $g$ divisor on $C$. The classical fact that $\smash{ \mathrm{Sym}^g C \to \mathrm{Pic}^g C }$ is birational implies that the restriction map $\Omega_C \to \Omega_C \otimes {\sh O}_D$ is an isomorphism on global sections. Indeed, $\Omega_C(-D)$ is a \emph{general} degree $g-2$ linear equivalence class, hence ineffective since $\smash{ \mathrm{Sym}^{g-2} C \to \mathrm{Pic}^{g-2} C }$ is not surjective for trivial dimension reasons, so we are done thanks to the exact sequence 
\begin{equation} 0 = \Gamma(\Omega_C(-D)) \longrightarrow \Gamma(\Omega_C) \longrightarrow \Gamma(\Omega_C \otimes {\sh O}_D) \end{equation}
because the last two terms have the same dimension $g$.

We now reinterpret the isomorphism $ \Gamma(\Omega_C) \to \Gamma(\Omega_C \otimes {\sh O}_D)$ above in deformation theoretic language. Let ${\sh N}_f$ be the normal sheaf of $f$, and consider the following commutative diagram with exact rows
\begin{center}
\begin{tikzpicture}
\matrix [column sep  = 6mm, row sep = 7mm] {
	\node (nww) {$0$}; &
	\node (nw) {$ T_{\mathbf{x}}C^g$}; &
	\node (nc) {$T_{\left( C,{\mathbf x},f \right)} \overline{\modu M}_{g,g}(V_t,\beta_t)$}; &
	\node (ne) {$T_{(C,f)} \overline{\modu M}_g(V_t,\beta_t)$};  &
	\node (nee) {$0$}; \\
	\node (sww) {$0$}; &
	\node (sw) {$ T_{\mathbf{x}}C^g $}; &
	\node (sc) {$T_{f({\mathbf x})}V^g_t $}; &
	\node (se) {$\Gamma(C,{\sh N}_f \otimes {\sh O}_D)$}; &
	\node (see) {$0$}; \\
};

\draw[->, thin] (nww) -- (nw);
\draw[->, thin] (nw) -- (nc);
\draw[->, thin] (nc) -- (ne);
\draw[->, thin] (ne) -- (nee);

\draw[->, thin] (sww) -- (sw);
\draw[->, thin] (sw) -- (sc);
\draw[->, thin] (sc) -- (se);
\draw[->, thin] (se) -- (see);

\draw[double equal sign distance, thin] (nw) -- (sw);
\draw[->, thin] (nc) -- (sc);
\draw[->, thin] (ne) -- (se);

\node at (-0.2,0) {$d_\Delta(\mathbf{ev})$};

\end{tikzpicture}
\end{center}
where $f({\mathbf x}) = (f(x_1),\ldots,f(x_g))$, and $V_t$ is of course the fiber of $V \to \Delta$ that $f$ maps to. As in the proof of Proposition \ref{prop: relatively smooth}, ${\sh N}_f \cong \Omega_C$ and the upper right term is naturally $\Gamma({\sh N}_f) = \Gamma(\Omega_C)$. Up to natural identifications, the right vertical map \emph{is} the map $\Gamma(\Omega_C) \to \Gamma(\Omega_C \otimes {\sh O}_D)$, so it must be an isomorphism. Then the central vertical map must be an isomorphism as well, which completes the proof of the claim at the beginning of this proof. 

Let ${\modu Z}$ be an irreducible/connected component of ${\modu V}$. Consider the $1$-morphism
\begin{equation} \Phi_g: \overline{\modu M}_{g,g}(V/\Delta,\beta;\theta) \longrightarrow \overline{\modu M}_g(V/\Delta,\beta;\theta) = \overline{\modu M} \end{equation}
which forgets the markings, cf. Remark \ref{rmk: compatible with forgetful}. Choose an irreducible component ${\modu W}$ of $\overline{\modu M}_{g,g}(V/\Delta,\beta;\theta)$ whose generic point is mapped by $\Phi_g$ to the generic point of ${\modu Z}$, seen in $\smash{|\overline{\modu M}|}$, and put the reduced structure on ${\modu W}$. It follows from the discussion above that the restriction of $\mathbf{ev}$ to ${\modu W}$ is dominant onto $V_\Delta^g$, and hence surjective since ${\modu W}$ is proper. In particular, $\smash{ \mathbf{ev}_o : {\modu W}_o \to V_o^g }$ is surjective. Then it is clear (and also implicit in \cite[\S4]{[BL99]}) that there exists $z \in {\modu W}_o(\cc)$, which is a simple type $\varpi$ stable map into $V_o$ (Definition \ref{defn: simple maps}), with $g$ markings in addition, for some $\varpi \in \Pi$. Indeed, for a general closed point in $\smash{V_o^g}$, all stable maps in $\smash{ \overline{\modu M}_{g,g}(V_o,\beta_o;\theta) }$ that evaluate to this $g$-tuple are of this form by Remark \ref{rmk: simple maps are unramified} and some elementary arguments. Let $z_0 = \Phi_g(z)$. On one hand, $z_0 \in {\modu V}[\varpi]({\mathbb C})$ by construction. On the other hand, we claim that $z_0 \in {\modu Z}(\cc)$ as well. Indeed, (1) the generic point of ${\modu W}$ obviously specializes to (contains in its closure) $z$, hence (2) the generic point of ${\modu Z}$ specializes to $z_0$ since specializations of points are preserved under continuous mappings, but on the other hand, (3) if ${\modu Z}'$ is the connected component of ${\modu V}$ which contains $z_0$, then its generic point trivially specializes to $z_0$, and so ${\modu Z}' = {\modu Z}$. We've thus shown that $|{\modu V}[\varpi]| \cap |{\modu Z}| \neq \emptyset$. Since ${\modu V}[\varpi]$ is connected and ${\modu Z}$ is a connected component of ${\modu V}$, it follows that $|{\modu V}[\varpi]| \subset |{\modu Z}|$, as desired.
\end{proof}

In the remainder of this section, we boil down the proof of the main theorem to Claim \ref{claim: key claim} below. We will only be able to prove Claim \ref{claim: key claim} in \S\ref{ssec: proof of main claim}, after a substantial detour.  

\begin{claim}\label{claim: key claim}
In Situation \ref{sit: main situation}, let $[\varpi_1\varpi_2]$ be an edge in $\Pi$, and let ${\modu Z}$ be an irreducible component of ${\modu V}$. Then ${\modu Z}$ contains ${\modu V}[\varpi_1]$ if and only if it contains ${\modu V}[\varpi_2]$. 
\end{claim}

\begin{proof}[Proof of Theorem \ref{thm: main theorem} assuming Claim \ref{claim: key claim}] By \eqref{eqn: total decomposition V}, it suffices to show that for general $p \in M(\cc)$, ${\modu V}_g(V_p,\beta;\theta)$ defined as in Remark \ref{rmk: connect to intro} is nonempty and irreducible. First, we prove non-emptiness. In Situation \ref{sit: main situation}, we have $\Pi \neq \emptyset$ by Remark \ref{rmk: partitions exist}, and hence ${\modu V} \neq \emptyset$ by Remark \ref{rmk: simple maps are unramified}. By Proposition \ref{prop: relatively smooth}, ${\modu V}_t \neq \emptyset$ for general $t \in \Delta$, and ${\modu V}_g(V_p,\beta;\theta) \neq \emptyset$ for general $p \in M$ follows easily, since $M$ is quasi-projective, so curves $\tau:\Delta \to M$ as is Situation \ref{sit: main situation} sweep out $M$ birationally. 

For irreducibility, we will rely on Claim \ref{claim: key claim}. First, we claim that ${\modu V}$ is irreducible in Situation \ref{sit: main situation}. Fix $\varpi_0 \in \Pi$, and let ${\modu Z}$ be any irreducible/connected component of ${\modu V}$. By Proposition \ref{prop: flexibility}, there exists $\varpi \in \Pi$ such that ${\modu Z}$ contains ${\modu V}[\varpi]$. By Proposition \ref{prop: graph is connected}, there exists a chain connecting $\gamma_0$ and $\gamma$ in $\Pi$. Applying Claim \ref{claim: key claim} inductively, we deduce that ${\modu Z}$ contains ${\modu V}[\varpi_0]$.  However, if ${\modu Z}'$ is another irreducible component of ${\modu V}$, then it must also contain ${\modu V}[\varpi_0]$ by the same token. By Proposition \ref{prop: relatively smooth}, ${\modu Z}={\modu Z}'$, so ${\modu V}$ is irreducible. 

Second, we claim that the fibers of ${\modu V} \to \Delta$ over all points in a dense open subset of $\Delta$ are irreducible. Of course, the rough idea is to eliminate the monodromy on the irreducible/connected components of the fibers. By \cite[Theorem 1]{[FP95]}, ${\modu V}$ has a quasi-projective coarse moduli space (using that $V \to \Delta$ is projective), so there certainly exists a $1$-morphism $\Delta' \to {\modu V}$ from a quasi-projective curve $\Delta'$ which intersects the central fiber set theoretically at a single point $o' \in \Delta'$. Pulling back along $\Delta' \to \Delta$, we are again in Situation \ref{sit: main situation}. The family ${\modu V}_{\Delta'} \to \Delta'$ now admits a section, and it follows from \cite[Proposition 2.2.1]{[Ro11]} that all fibers of ${\modu V}_{\Delta'} \to \Delta'$ are connected. Indeed, in the language of \cite{[Ro11]}, the `c.c.o.' along the section is an open connected substack of ${\modu V}_{\Delta'}$ with connected fibers, and it has to be all of ${\modu V}_{\Delta'}$ by the previous step. In conclusion, the fibers of ${\modu V}_{\Delta'} \to \Delta'$ are irreducible by Proposition \ref{prop: relatively smooth} and \cite[Proposition 4.16]{[DM69]}.  

Finally, we note that, since $M$ is quasi-projective, curves $\tau:\Delta \to M$ as is Situation \ref{sit: main situation} sweep out $M$ birationally, so we're done. To be very precise, pick such a sweeping family and consider the pullback of ${\modu V}_g(V_M/M,\beta;\theta)$ to the total space of this family, and argue using e.g. \cite[Theorem 4.17.(i)]{[DM69]}.
\end{proof}

The plan for the rest of the paper is the following. In \S\ref{sec: quasi-traceless covers section}, we discuss a property of (ramified) covers of elliptic curves which will come up in our criterion that characterizes primitive class stable maps to $E \times F$ that deform to nearby surfaces (Theorem \ref{thm: necessary condition to deform}). In \S\ref{ssec: bubbling up} we carry out the rest of the work needed to prove this criterion, and in \S\ref{ssec: proof of main claim} we use it to prove Claim \ref{claim: key claim}. 

\section{Quasi-traceless covers of elliptic curves}\label{sec: quasi-traceless covers section}

\subsection{Definition and the relation to the Atiyah surface}\label{ssec: quasi-traceless section intro and Atiyah surface} Let $E$ be a (smooth connected projective) genus $1$ curve over $\cc$ and $f:C \to E$ a nonconstant map from a (smooth connected projective) curve $C$. We have a pullback map $\hh^1({\sh O}_E) \to \hh^1({\sh O}_C)$ which we may construct as either the composition of $\hh^1({\sh O}_E) \to \hh^1(f_*{\sh O}_C)$ with the map $\hh^1(f_*{\sh O}_C) \to \hh^1({\sh O}_C)$ (which is an isomorphism in our case) coming from the Leray spectral sequence, or equivalently, as the differential at the origin of the pullback map $f^*:\mathrm{Pic}(E) \to \mathrm{Pic}(C)$. The \emph{trace map} on $1$-forms
$$ \mathrm{Tr}: \Gamma(\Omega_C) \longrightarrow \Gamma(\Omega_E) $$
is, by definition, the Serre dual of $\hh^1({\sh O}_E) \to \hh^1({\sh O}_C)$. Let $r \in C(\cc)$. 

\begin{defn}\label{defn: quasi-traceless}
We say that $f:C \to E$ is \emph{quasi-traceless} (relative to the point $r \in C$) if $\mathrm{Tr}(\eta) = 0$, for all regular $1$-forms $\eta \in \Gamma(\Omega_C)$ such that $\eta(r) = 0$. 
\end{defn}

By the classical results of \cite{[At57]}, there exists a unique rank $2$ bundle ${\sh V}$ on $E$ that can be obtained as a nonsplit extension
\begin{equation}\label{eqn: Atiyah sequence}
0 \longrightarrow {\sh O}_E \longrightarrow {\sh V} \longrightarrow {\sh O}_E \longrightarrow 0.
\end{equation} 
Recall that ${\sh V}$ is self-dual. Let $\varsigma:S = \mathrm{Proj}_E\mathrm{Sym} \mathop{ }{\sh V} \to E$ be the corresponding ruled surface, which we will sometimes call the Atiyah surface, and let $E_\infty \subset S$ be the distinguished section corresponding to the unique copy of ${\sh O}_E$ inside ${\sh V}$.

\begin{rmk}\label{rmk: no complete curves on open Atiyah surface}
The quasiprojective surface $S \backslash E_\infty$ contains no complete curves \cite[\S2]{[Za19a]}. (However, this is false in positive characteristic.)
\end{rmk}

We introduce some jargon (applicable in any abelian category), which will be convenient to use in the proof of Lemma \ref{lem: connection to Atiyah surface} below. 

\begin{defn}\label{defn: homomorphism annihilates extension}
We say that a homomorphism $A \to A'$ \emph{annihilates} an extension $0 \to A \to B \to C \to 0$ if the pushout of the extension along the homomorphism is a split extension. Equivalently, the homomorphism $\mathrm{Ext}^1(C,A) \to \mathrm{Ext}^1(C,A')$ induced by $A \to A'$ annihilates the extension. 
\end{defn}

\begin{lem}\label{lem: connection to Atiyah surface}
Let $C$ be a (smooth connected complex projective) curve, $r \in C(\cc)$, and $f:C \to E$ a nonconstant map. If there exists a morphism $\smash{ \widetilde{f}:C \to S }$ such that $\smash{ f = \varsigma \widetilde{f} }$ and $\smash{ \widetilde{f}^{-1}(E_\infty) = \{r\} }$ scheme-theoretically, then $f:C \to E$ is quasi-traceless relative to $r$.
\end{lem}

\begin{proof} Note that $\iota_*{\sh O}_{E_\infty} \otimes {\sh O}_S(E_\infty) \cong \iota_*{\sh O}_{E_\infty}$, where $\iota:E_\infty \hookrightarrow S$, since the normal bundle of $E_\infty$ in $S$ is trivial \cite[\S2]{[Za19a]}. Then we have a short exact sequence
\begin{equation}\label{eqn: some ses on Atiyah surface}
0 \longrightarrow {\sh O}_S \longrightarrow {\sh O}_S(E_\infty) \longrightarrow \iota_*{\sh O}_{E_\infty} \longrightarrow 0.
\end{equation}
We have $h^1({\sh O}_S) = 1$ since $S$ is birational to $E \times {\mathbb P}^1$, and $h^0({\sh O}_S(E_\infty) ) = 1$ by \cite[Proposition 2.2]{[Za19a]}. Then the first 4 (nonzero) terms in the long exact sequence
$$ 0 \to \Gamma({\sh O}_S) \to \Gamma({\sh O}_S(E_\infty)) \to \Gamma({\sh O}_{E_\infty}) \to \mathrm{H}^1({\sh O}_S) \xrightarrow{\psi} \mathrm{H}^1({\sh O}_S(E_\infty)) \to \cdots $$ 
have dimension $1$, so $\psi = 0$. Equivalently, $\mathrm{Ext}^1({\sh O}_S,{\sh O}_S) \to \mathrm{Ext}^1({\sh O}_S,{\sh O}_S(E_\infty))$ induced by ${\sh O}_S \to {\sh O}_S(E_\infty)$ is identically $0$. Therefore, in the language of Definition \ref{defn: homomorphism annihilates extension}, ${\sh O}_S \to {\sh O}_S(E_\infty)$ annihilates the pullback of \eqref{eqn: Atiyah sequence} along $\varsigma$. Pulling back to $C$, we deduce that ${\sh O}_C \to {\sh O}_C(r)$ annihilates the pullback of \eqref{eqn: Atiyah sequence} along $f$. Unwinding the definitions, the composition $\hh^1({\sh O}_E) \to \hh^1({\sh O}_C) \to \hh^1({\sh O}_C(r))$ is identically zero, and hence $f:C \to E$ is quasi-traceless relative to $r \in C$ by Serre duality. 
\end{proof}

The converse of Lemma \ref{lem: connection to Atiyah surface} is true, but we will not need it. 

\subsection{Moduli of quasi-traceless covers of genus $2$}\label{ssec: moduli of quasi-traceless covers for genus 2} Let $E$ be a (smooth, connected, projective) genus $1$ curve over $\cc$. In this section, we construct a partial compactification of the moduli space of quasi-traceless covers of $E$ with sources of genus $2$, and prove that, after fixing the maximal factorization through isogenies (\S\ref{sss: factorizations through isogenies of elliptic curves}), the moduli space is irreducible. Although the construction makes sense even for higher genus, we strongly believe that a better compactification exists in general, so we will limit ourselves to the the genus $2$ case, and hopefully discuss the generalization in future work. For this reason, this subsection is rather ad hoc and technical. 

Let $\smash{ \overline{\modu M}_2(E,d) }$ be the moduli space of genus $2$ degree $d$ stable maps to $E$. Let $\smash{ {\modu K} \subset \overline{\modu M}_2(E,d) }$ be the open substack consisting of stable maps $f:C \to E \times S$ which are finite and have sources of compact type. Let $\smash{ {\modu K}_1 \subset \overline{\modu M}_{2,1}(E,d) }$ be the preimage of ${\modu K}$ under the forgetful $1$-morphism $\smash{ \overline{\modu M}_{2,1}(E,d) \to \overline{\modu M}_2(E,d) }$. (Note that it is not true that all stable maps in ${\modu K}_1$ are finite.) The upshot of the construction is that our partial compactification $\smash{ \widetilde{\modu Q} \subset {\modu K}_1 }$ will be defined as the degeneracy locus of the natural ${\sh O}_{{\modu K}_1}$-module homomorphism
\begin{equation}\label{eqn: trace + eval} {\mathbb E} \xrightarrow{\gamma \oplus \mathrm{tr}}  {\sh L} \oplus {\sh O}_{{\modu K}_1}, \end{equation}
where ${\mathbb E}$ and ${\sh L}$ are the Hodge bundle and the tautological cotangent line bundle at the marking respectively. The definitions of the bundles and maps involved are briefly reviewed below. 

First, we discuss the trace map on ${\modu K}$. Let $f:C \to E \times S$ be a stable map in ${\modu K}(S)$. In this paragraph, we assume that $S$ is (at least locally) noetherian, and we will explain later why the assumption was harmless. Let $\pi:C \to S$ be the projection map, and $\eta:E \times S \to S$ the projection to the second factor. Consider the ${\modu O}_S$-module homomorphism
\begin{equation}\label{eqn: relative dual trace} {\sh O}_S \cong  \hh^1({\sh O}_E) \otimes {\sh O}_S \cong {\mathrm R}^1\eta_*{\sh O}_{S \times E} \longrightarrow {\mathrm R}^1\pi_*{\sh O}_C \cong (\pi_*\omega_{C/S})^\vee, \end{equation}
obtained as the composition of ${\mathrm R}^1\eta_*{\sh O}_{S \times E} \to {\mathrm R}^1\eta_*(f_*{\sh O}_C)$ induced by $f^\#:{\sh O}_{S \times E} \to f_*{\sh O}_C$ with the `edge map' ${\mathrm R}^1\eta_*(f_*{\sh O}_C) \to {\mathrm R}^1\pi_*{\sh O}_C$ from the Grothendieck spectral sequence \cite[Ch. 0, (12.2.4) and (12.2.5.1)]{[EGAIII1]}. Then ${\mathrm R}^1\eta_*(f_*{\sh O}_C) \to {\mathrm R}^1\pi_*{\sh O}_C$ is in fact an isomorphism because $f$ is finite by assumption, so ${\mathrm R}^1f_*{\sh O}_C = 0$. The last canonical isomorphism in \eqref{eqn: relative dual trace} is a fairly elementary case of coherent duality, please see \cite[Theorem (21)]{[Kl80]}. We sketch the proof that the formation of \eqref{eqn: relative dual trace} commutes with base change. The formation of $f_*{\sh O}_C$, as well as the formation of $\mathrm{R}^1\eta_*f_*{\sh O}_C$ commutes with base change \cite[(6.9.9.2)]{[EGAIII2]}. (We are using a very common version of the cohomology and base change theorem, namely the stronger version of \cite[II, \S5, Corollary 3]{[Mu85]} which states that cohomology commutes with base change, not just with passing to fibers. This version still follows from \cite[(6.9.9.2)]{[EGAIII2]}.) Hence all terms involved commute with base change. The map ${\mathrm R}^1\eta_*{\sh O}_{S \times E} \to {\mathrm R}^1\eta_*(f_*{\sh O}_C)$ commutes with base change because the `cohomology and base change' morphism is always functorial in the sheaf -- regardless of whether it is an isomorphism or not. The edge map also commutes with base change, and we assume this must be well-known to experts, although we do not know a general statement or a reference. In our situation, it is possible to proceed as follows. If the base change map $S' \to S$ is affine and $S$ and $S'$ are separated, then, choosing an affine open cover of $E \times S$, and taking advantage of the fact that both $f$ and $E \times S' \to E \times S$ are affine, it is possible to check the claim by brute force using \v{C}ech cohomology in the form of \cite[Ch. III, Proposition 8.7]{[Ha77]}. However, the claim is local on both $S$ and $S'$, so we could have even assumed that $S$ and $S'$ are affine schemes. (A more conceptual reason why \eqref{eqn: relative dual trace} ought to commute with base change is given by \cite[Proposition 1.3.(b) and (c)]{[LLR04]}.) 

In conclusion, we obtain an ${\modu O}_{\modu K}$-module homomorphism from ${\modu O}_{\modu K}$ to the dual of the Hodge bundle over ${\modu K}$; dualizing, we obtain the trace map on ${\modu K}$, and pulling back we obtain the trace map on ${\modu K}_1$
\begin{equation}\label{eqn: relative trace}
\mathrm{tr}: {\mathbb E} \longrightarrow {\mathscr O}_{{\modu K}_1}.
\end{equation}
(We are implicitly using the well-known fact that forgetting the marked points and stabilizing does not change the Hodge bundles \cite[Exercise 25.3.1]{[mirror-symmetry]}.) We should also clarify that the conceptual issues caused by the noetherian hypothesis in the cohomology and base change theorem can be circumvented using the alternate definition of quasi-coherent sheaves on Deligne-Mumford stacks in \cite[Ch. XIII, \S2, starting on p. 336]{[ACG11]}. However, knowing that \eqref{eqn: relative dual trace} commutes with more than just \'{e}tale/flat base change will be necessary later in this section. 

Second, if $(f:C \to E \times S,\rho: E \times S \to C)$ is an object in ${\modu K}_1(S)$, with projection $\pi:C \to S$, the adjoint of the identity map $\omega_{C/S} \to \pi^* \rho^* \omega_{C/S}$ is 
\begin{equation}\label{eqn: hodge to taut over S}
\pi_* \omega_{C/S} \longrightarrow \rho^* \omega_{C/S}.
\end{equation}
It is not hard to check that the formation of \eqref{eqn: hodge to taut over S} commutes with base change. Then \eqref{eqn: hodge to taut over S} induces the `evaluation at the marking' ${\modu O}_{{\modu K}_1}$-module homomorphism 
\begin{equation}\label{eqn: hodge to tautological}
\gamma: {\mathbb E} \longrightarrow {\sh L}.
\end{equation}
In conclusion, \eqref{eqn: relative trace} and \eqref{eqn: hodge to tautological} complete the construction of \eqref{eqn: trace + eval}. We take $\widetilde{\modu Q}$ to be the degeneracy locus of \eqref{eqn: trace + eval}, that is, the closed substack defined by the induced ideal sheaf ${\sh H}\!om({\sh L}, \det {\mathbb E}) \hookrightarrow {\sh O}_{{\modu K}_1}$. This completes our rather ad hoc construction of the partially compactified moduli space. 

If $\smash{ {\modu M}_{2,1}^{\mathrm{sm}}(E,d) \subset {\modu K}_1}$ is the open substack of stable maps with smooth sources, we take $\smash{ {\modu Q} }$ to be the restriction of $\smash{\widetilde{\modu Q} }$ to $\smash{ {\modu M}_{2,1}^{\mathrm{sm}}(E,d) }$, and we have a closed immersion
\begin{equation}\label{eqn: Q closed immersion}
{\modu Q} \hookrightarrow  {\modu M}_{2,1}^{\mathrm{sm}}(E,d).
\end{equation}

\begin{rmk}\label{rmk: testing quasi-traceless on complex points}
If $S$ is reduced and of finite type over $\cc$, and $(f:C \to E \times S, \rho:E \times S \to C)$ is a stable map in $\smash{ {\modu M}_{2,1}^{\mathrm{sm}}(E,d) (S)}$, then, in order to show that this is an object of ${\modu Q}(S)$, it suffices to check that all fibers over $\cc$-points of $S$ are quasi-traceless in the sense of Definition \ref{defn: quasi-traceless}. 
\end{rmk}

Throughout much of this section, we will be working locally at the stable map in $\smash{\widetilde{\modu Q}(\cc)}$ described below. Let $q \in E(\cc)$ and $\Lambda_1,\Lambda_2 \subseteq \mathrm{H}_1(E,\zz)$ two sublattices whose indices add up to $d$. The \emph{degenerate quasi-traceless cover} corresponding to $q,\Lambda_1,\Lambda_2$ is the $1$-marked stable map $h: E_1 \cup {\mathbb P}^1 \cup E_2 \to E $ (the rational component contains the marking $x$, and is in the middle of the chain of 3 components) which contracts the rational component to $q \in E (\cc)$ and restricts on $E_i$ to the isogeny $E_i \to E$ corresponding to $\Lambda_i$, for $i=1,2$. Let $\smash{ y \in {\modu K}_1(\cc) }$ be the corresponding point; it is clear that $\smash{ y \in \widetilde{\modu Q}(\cc) }$ because $\gamma(y) = 0$. We begin with some local calculations. 

\begin{lem}\label{lem: quasi-traceless specific local calculation}
If $\mathrm{ev}:\overline{\modu M}_{2,1}(E,d) \to E$ is the evaluation map, then:
\begin{enumerate}
\item $\mathrm{ev}$ is smooth of relative dimension $2$ at $y$; and 
\item $y$ is an \emph{infinitesimally isolated} vanishing point of $\gamma|\mathrm{ev}^{-1}(q)$.
\end{enumerate}
\end{lem}

To clarify, the second part means that $y^*\gamma: y^*{\mathbb E} \to y^*{\sh L}$ is identically zero, and if $\widetilde{y}:\spec \cc[\epsilon]/(\epsilon^2) \to \mathrm{ev}^{-1}(q)$ restricts to $y$ on $\spec \cc \hookrightarrow \spec \cc[\epsilon]/(\epsilon^2)$ and $\widetilde{y}^*\gamma: \widetilde{y}^*{\mathbb E} \to \widetilde{y}^*{\sh L}$ is identically zero, then $\widetilde{y} = y \circ (\spec \cc[\epsilon]/(\epsilon^2) \to \spec \cc)$.

\begin{proof} Let $\smash{H = E_1 \cup {\mathbb P}^1 \cup E_2}$ be the source of $h$. The space of first order deformations of the stable marked curve $(H,x)$ is $\smash{\mathrm{Ext}^1(\Omega_H(x),{\sh O}_H) =\hh^1(T_H^\mathrm{log}(-x))}$, where $\smash{ T_H^\mathrm{log} = \omega_H^\vee }$. Let $\smash{ \mathrm{Def}^1 }$ and $\mathrm{Obs}$ be the spaces of first order deformations and respectively obstructions of $h$ (with the marking $x$), that is,
$$ \mathbb{E}\mathrm{xt}^i(h^*\Omega_E \to \Omega_H(x), {\sh O}_H) =
\begin{cases}
\mathrm{Def}^1, & i=0 \\
\mathrm{Obs}, & i=1 \\
\end{cases}$$ with $h^*\Omega_E \to \Omega_H(x)$ in degrees $0,1$ \cite{[Be97], [BF97]}. We have an exact sequence
\begin{equation}\label{eqn: deformation theoretic long exact sequence in qt section}
0 \to \Gamma({\sh O}_H) \to \mathrm{Def}^1 \to \hh^1(T_H^\mathrm{log}(-x)) \to \hh^1({\sh O}_H) \to \mathrm{Obs} \to 0,
\end{equation}
in which the third nonzero map is the map on $\hh^1$ induced by the composition of $T_H^\mathrm{log}(-x) \hookrightarrow T_H^\mathrm{log}$ with the differential $\mathrm{d}h: T_H^\mathrm{log} \to h^* T_E = {\sh O}_H$. This composition also corresponds to an element of $\Gamma(\omega_H(x))$. The third nonzero map in \eqref{eqn: deformation theoretic long exact sequence in qt section} is surjective because it is Serre dual to the map $\Gamma(\omega_H) \to \Gamma(\omega_H^{\otimes 2}(x))$, which can be proved injective directly, using that $\omega_H|{\mathbb P}^1 \approx {\sh O}_{{\mathbb P}^1}$ and $\omega_H|E_i \approx {\sh O}_{E_i}(E_i \cap {\mathbb P}^1)$. It follows that $\mathrm{Obs} = 0$ and $\dim_\cc \mathrm{Def}^1 = 3$, so $\overline{\modu M}_{2,1}(E,d)$ is smooth of dimension $3$ at $y$. Moreover, since $\Gamma({\sh O}_H) = \Gamma(h^*T_E)$ appears in \eqref{eqn: deformation theoretic long exact sequence in qt section} as the space of first order deformations of $h$ that don't change the source $H$, the first nonzero map in \eqref{eqn: deformation theoretic long exact sequence in qt section} composed with $\mathrm{d}(\mathrm{ev}):\mathrm{Def}^1 \to T_qE$ is the canonical identification $\Gamma({\sh O}_H) = \Gamma(h^*T_E) = T_qE$, so $\mathrm{d}(\mathrm{ev})$ is injective, which completes the proof of the first part. 

Let $\smash{ \mathrm{Def}^1_q \subset \mathrm{Def}^1 }$ be the subspace of first order deformations contained in $\smash{ \mathrm{ev}^{-1}(q) }$. The paragraph above shows that $\smash{ \mathrm{Def}^1_q }$ can also be identified with the cokernel of the first nonzero map in \eqref{eqn: deformation theoretic long exact sequence in qt section} and that $\smash{ \mathrm{ev}^{-1}(q) \hookrightarrow \overline{\modu M}_{2,1}(E,d) \to {\mathfrak M}_{2,1} }$ is unramified at $y$, that is, the corresponding map $\smash{ \mathrm{Def}^1_q \to \hh^1(T_H^\mathrm{log}(-x)) }$ on first order deformations at $y$ is injective. Let $\smash{ W \subset \hh^1(T_H^\mathrm{log}(-x)) }$ be the ($2$-dimensional) subspace of first order deformations of $(H,x)$ that preserve the two singularities. Then the second part of the lemma follows from the next two claims, whose verification is left to the reader. First, the image of the map $\smash{ \mathrm{Def}^1_q \to \hh^1(T_H^\mathrm{log}(-x)) }$ is transverse to $W$ (this follows from some more calculations on \eqref{eqn: deformation theoretic long exact sequence in qt section}). Second, if $\gamma$ (now considered on $\smash{ \overline{\modu M}_{2,1} }$ by abuse of language) vanishes identically on a first order deformation $\smash{ v \in \hh^1(T_H^\mathrm{log}(-x)) }$ of $(H,x)$, then $v \in W$. 
\end{proof}

\begin{prop}\label{prop: quasi-traceless local calculation}
The moduli stack $\smash{\widetilde {\modu Q}}$ is smooth of local dimension $2$ at $y$. 
\end{prop} 

\begin{proof}
In light of Lemma \ref{lem: quasi-traceless specific local calculation} and the easily checked fact that $\mathrm{tr}(y) \neq 0$, the slightly stronger fact that $\smash{ \mathrm{ev}|\widetilde {\modu Q} }$ is smooth of relative local dimension $1$ at $y$ follows from Lemma \ref{lem: quasi-traceless general local calculation} below, stated separately (with independent notation) for clarity.

\begin{lem}\label{lem: quasi-traceless general local calculation}
Let ${\modu X}$ and ${\modu Y}$ be finite type Deligne-Mumford stacks over $\cc$, $p \in {\modu X}(\cc)$ a $\cc$-point, $f:{\modu X} \to {\modu Y}$ a $1$-morphism, and ${\sh E},{\sh L}_1,{\sh L}_2$ locally free quasi-coherent ${\sh O}_{\modu X}$-modules of ranks $2$, $1$, and $1$ respectively, with a ${\sh O}_{\modu X}$-module homomorphism 
$$ \phi = \phi_1 \oplus \phi_2: {\sh E} \longrightarrow {\sh L}_1 \oplus {\sh L}_2, $$
such that the following conditions are satisfied:
\begin{enumerate}
\item ${\modu Y}$ is smooth over $\spec \cc$, and $f$ is smooth of relative dimension $2$ at $p$;
\item $p$ is an infinitesimally isolated vanishing point of $\phi_1|f^{-1}(p)$;
\item $\phi_2$ doesn't vanish at $p$, that is, $p^*\phi_2: p^*{\sh E} \to p^*{\sh L}_2$ is nonzero.
\end{enumerate}
Then, if ${\modu Z} \subset {\modu X}$ is the degeneracy locus of $\phi$, the restriction $f|{\modu Z}$ is smooth of local relative dimension $1$ at $p$. 
\end{lem}

\begin{proof}
The statement is \'{e}tale local both on the target and the source, so we may assume that ${\modu X}$, ${\modu Y}$ are schemes (so we will call them $X$ and $Y$), and that ${\sh E}$, ${\sh L}_1$, and ${\sh L}_2$ are all trivial vector bundles. Let 
$$ \phi = \begin{bmatrix} \phi_1 \\ \phi_2 \end{bmatrix} = \begin{bmatrix} a & b \\ c & d \end{bmatrix}, \quad \text{where} \quad a,b,c,d \in \Gamma({\sh O}_X). $$ 
The first assumption implies that $\smash{ {\widehat{\sh O}}_{X,p} \cong  {\widehat{\sh O}}_{Y,f(p)}[\![x,y]\!] } $. The degeneracy locus is the vanishing locus of $\det \phi= ad-bc$, and the observation that
\begin{equation*}
\begin{bmatrix}
\partial_x \det \phi \\
\partial_y \det \phi \\
\end{bmatrix}
=
\underbrace{\begin{bmatrix}
\partial_x a & \partial_x b \\
\partial_y a & \partial_y b \\
\end{bmatrix}
\begin{bmatrix}
d \\
-c \\
\end{bmatrix} }_{\neq {0 \brack 0}\text{ at }p}
+
\underbrace{\begin{bmatrix}
\partial_x c & \partial_x d \\
\partial_y c & \partial_y d \\
\end{bmatrix}
\begin{bmatrix}
-b \\
a \\
\end{bmatrix}}_{= {0 \brack 0} \text{ at }p}
\neq
\begin{bmatrix}
0 \\
0 \\
\end{bmatrix} \text{ at } p
\end{equation*}
comples the proof. \end{proof}

As explained above, Lemma \ref{lem: quasi-traceless general local calculation} completes the proof of Proposition \ref{prop: quasi-traceless local calculation}. 
\end{proof}

Recall the discussion in \S\ref{sss: factorizations through isogenies of elliptic curves}. The decomposition \eqref{eqn: total decomposition elliptic curve} induces decompositions 
\begin{equation}\label{eqn: Q decomp} {\modu Q} = \bigsqcup_{[\hh_1(E,\zz):\Lambda] | d} {\modu Q}(\Lambda) \quad \text{and} \quad \widetilde{\modu Q} = \bigsqcup_{[\hh_1(E,\zz):\Lambda] | d} \widetilde{\modu Q}(\Lambda), \end{equation}
where ${\modu Q}(\Lambda) \subseteq \widetilde{\modu Q}(\Lambda) \subseteq \overline{\modu M}_{2,1}(E,d;\Lambda)$.

\begin{thm}\label{thm: quasi-traceless final theorem}
The moduli space ${\modu Q}(\Lambda_1+\Lambda_2)$ is irreducible of dimension $2$ and contains the point $y$ in its closure relative to $\smash{{\modu K}_1}$, namely $\smash{ \widetilde{\modu Q}(\Lambda_1+\Lambda_2) }$. 
\end{thm}

Informally, the picture is the following: $\smash{ \widetilde{\modu Q}(\Lambda_1+\Lambda_2) }$ is $2$-to-$1$ over the `main component' of $\smash{ \overline{\modu M}_2(E,d;\Lambda_1+\Lambda_2) }$ via the forgetful morphism, and ramified at $y$. The proof below could be described as showing that it is `totally ramified' at $y$, with no reference to it being $2$-to-$1$.

\begin{proof}
By Proposition \ref{prop: quasi-traceless local calculation}, proving that all irreducible components of $\smash{ \widetilde{\modu Q}(\Lambda_1+\Lambda_2) }$ contain $y$ solves both claims at once, cf. \cite[Proposition 4.16]{[DM69]}. Let $|{\modu Y}|$ be an irreducible component of $\smash{ |\widetilde{\modu Q}(\Lambda_1+\Lambda_2)| }$, and $\smash{ \overline{\modu Y} }$ its closure in $\overline{\modu M}_2(E,d;\Lambda_1+\Lambda_2)$, say, with the reduced structure. Consider the forgetful morphism
\begin{equation}
\phi: \overline{\modu M}_{2,1}(E,d;\Lambda_1+\Lambda_2) \longrightarrow \overline{\modu M}_2(E,d;\Lambda_1+\Lambda_2),
\end{equation}
cf. \S\ref{sss: factorizations through isogenies of elliptic curves}, and Remark \ref{rmk: compatible with forgetful}. It is completely routine to show that $\phi$ maps the generic point of $|{\modu Y}|$ to the generic point of some irreducible component of $\overline{\modu M}_2(E,d;\Lambda_1+\Lambda_2)$ that generically parametrizes stable maps with smooth sources. However, it is well-known that $\overline{\modu M}_2(E,d;\Lambda_1+\Lambda_2)$ has a unique irreducible component that generically parametrizes stable maps with smooth sources, and let $\mu$ be its generic point. (All of \cite{[Ka03], [Bu15], [GK87]} contain much stronger results than the stated one. Strictly speaking, the references above treat the `primitive' case $\Lambda_1 + \Lambda_2 = \hh_1(E,\zz)$ and in slightly different language, but it is easy to reduce to this case. Alternatively, our claim also follows from elementary abelian surface theory.) Thus $\mu$ belongs to the image of $\smash{ \phi|{\overline{\modu Y}} }$. However, $\mu$ clearly specializes to $\phi(y)$, and $\smash{ \phi|{\overline{\modu Y}} }$ is proper, so there exists a $\cc$-point $\smash{ y_0 \in \overline{\modu Y}(\cc) }$ such that $\phi(y_0) = \phi(y)$. However, the only point in $\smash{ \widetilde{\modu Q}(\Lambda_1+\Lambda_2)(\cc) }$ (or its closure) with this property is $y_0 = y$, completing the proof.
\end{proof}

\section{Deformable stable maps to the product surface}\label{sec: deformable stable maps}

\subsection{Introduction}\label{ssec: introduction to deformable maps section} We return to the setup of \S\ref{sec: families of severi varieties}, and in particular Situation \ref{sit: main situation}. A \emph{geometric base change} $(\Delta',o') \to (\Delta,o)$ (or just `base change' if no confusion is possible) is a map $\delta:\Delta' \to \Delta$ from a smooth connected quasi-projective curve $\Delta'$, such that $\delta^{-1}(o) = \{o'\}$ set-theoretically. If $\delta^{-1}(o) = mo'$ scheme-theoretically, we say that $\delta$ has local degree $m$. 

A word on notation: to avoid overloading the letter `C', from now on, we will often drop the source of a stable map and specify only the map. For instance, we could write just $f$ instead of $(C,f)$, and $[f]$ to refer to the corresponding point in the moduli space. 

\begin{defn}\label{defn: deformable}
In Situation \ref{sit: main situation}, we say that a stable map $h$, $[h] \in \overline{\modu M}_g(V_o,\beta_o)(\cc)$, is \emph{deformable} if there exists a geometric base change $(\Delta',o') \to (\Delta,o)$ and a stable map $f$, $[f] \in \overline{\modu M}_g(V/\Delta,\beta)(\Delta')$ i.e. a stable map over $\Delta' \to \Delta$, such that $f_{o'} \approx h$.
\end{defn}

\begin{rmk}\label{rmk: deformable alternative}
Although Definition \ref{defn: deformable} looks rather naive, it is not different from alternate formulations. Note that if there exists a stable map $[\hat{f}] \in \overline{\modu M}_g(V/\Delta,\beta)(k)$ for some $\spec k \to \Delta$ mapping to the generic point of $\Delta$ such that the point $\smash{ [\hat{f}] }$ specializes to $[h]$ (i.e. contains $[h]$ in its closure in $|\overline{\modu M}_g(V/\Delta,\beta)|$), then $h$ is actually deformable in the sense of Definition \ref{defn: deformable} above. The only substantial ingredient needed to prove this is that $\overline{\modu M}_g(V/\Delta,\beta)$ is of finite type, and $\overline{M}_g(V/\Delta,\beta) \to \Delta$ is projective \cite[Theorem 1.1]{[FP95]}, since $V \to \Delta$ is projective. 
\end{rmk}

We can now state the deformability criterion mentioned in the introduction. 

\begin{thm}\label{thm: necessary condition to deform}
Consider a stable map $h$, $[h] \in {\modu T}_g(V/\Delta,\beta)_o(\mathbb C)$, with no contracted components of strictly positive genus. If $h$ is deformable (cf. Definition \ref{defn: deformable}), then
\begin{enumerate}
\item any positive genus irreducible component $C$ of the source other than the `backbone' is a leaf of the dual tree; and, moreover
\item for any $C$ as above, if $r \in C$ is the sole node of the source of $h$ on $C$, then the restriction of $h$ to $C$ composed with $V_o \to F$ is quasi-traceless onto $F$ relative to $r$, in the sense of Definition \ref{defn: quasi-traceless}.
\end{enumerate}
(The `backbone' is the unique irreducible component of the source that maps isomorphically onto a fiber of $V_o \to F$.) 
\end{thm}

It is quite remarkable that the criterion is nontrivial even in the case $g=d+1$, when the Severi varieties are completely trivial. 

The technique to prove Theorem \ref{thm: necessary condition to deform} is similar to the iterative blowup of the elliptic fibers in \cite{[Ch02]}. The main difference is that we will have to consider not only `chain-like expansions' \cite[Figure 2]{[Ch02]}, but also the more general `tree-like expansions' in Figure A below. The paragraph \S\ref{ssec: algebraic bubbling up} below contains the algebraic fact underlying the semistable reduction procedure in \S\ref{ssec: bubbling up}. It may be skipped and referred back to as necessary -- it will become much easier to read once its geometric meaning becomes clear in the proof of Lemma \ref{lem: bubbling up step}.

\subsubsection{Algebraic preliminaries}\label{ssec: algebraic bubbling up} Let $A$ be an integral domain of characteristic $0$ and $R = A [\![t ]\!]$ the ring of formal power series in $A$. Fix $m \in \zz_+$, and
$$ W_m(A) = \left\{ p(t,x) = x^m + \sum_{j=1}^m \alpha_j(t) x^{m-j} \in R[x]: \alpha_j(0)=0 \text{ for all } j \right\}. $$
First, for $p = p(t,x) \in W_m(A)$, let
\begin{equation}
\nu(p) =\begin{cases}
t^{-m}p(t,tx) & \text{ if } t^{-m}p(t,tx) \in R[x], \\
\text{[lose]} & \text{ otherwise.} \\
\end{cases}
\end{equation}
Second, for $q = q(t,x) \in R[x]$, let
\begin{equation}
\tau(q) =\begin{cases}
q(t,x+\alpha) & \text{ if there exists $\alpha \in A$ such that }q(t,x+\alpha) \in W_m(A), \\
\text{[win]} & \text{ otherwise. } \\
\end{cases}
\end{equation}
Note that the $\alpha$ in the definition of $\tau$ is unique if it exists. 

\begin{lem}\label{lem: eventually lemma}
Let $p(t,x) \in W_m(A)$ which is not a perfect $m$-th power. Then there exists a positive integer $\mu$ such that the sequence $ r, \nu(r), \tau(\nu(r)), \nu(\tau(\nu(r))), 
\ldots $ 
starting with $r(t,x) = p(t^\mu,x)$ eventually ends with $\mathrm{[win]}$.
\end{lem}

\begin{proof} 
This is a variation on well-known ideas, so we'll be brief. First, note that if the lemma holds for some extension $A' \supset A$, then it also holds for $A$, so we may reduce to the case when $A = {\mathbb K} = \overline{\mathbb K}$ is an algebraically closed field. By the Newton-Puiseux Theorem and the fact that ${\mathbb K}[\![t]\!]$ is integrally closed, there exist a positive integer $\mu$ and $\phi_1,\ldots,\phi_m \in {\mathbb K}[\![t]\!]$ such that $$  p(t^\mu,x) = \left(x-\phi_1(t) \right) \cdots \left(x-\phi_m(t) \right). $$  
Then it can be shown that $(\tau \nu)^N(r) = \mathrm{[win]}$, where $N$ is minimal with the property that not all $\phi_1,\ldots,\phi_m$ have the same degree $N$ monomial. 
\end{proof}

\subsection{Bubbling up}\label{ssec: bubbling up} Let $T={\mathbb P}^1 \times F$ be the trivial ruled surface over $F$, and $S$ the Atiyah ruled surface over $F$ cf. \S\ref{ssec: quasi-traceless section intro and Atiyah surface}, respectively. ($T$ and $S$ are diffeomorphic, but not isomorphic.)

\begin{defn}\label{defn: drawer}
A \emph{drawer} is a (reduced but typically reducible) surface $X_o$ with simple normal crossing singularities whose dual simplicial complex is a tree, together with a morphism $\xi:X_o \to V_o = E \times F$ such that the following hold: 
\begin{enumerate}
\item [(i)] There exists a distinguished irreducible component $Y_\star$ of $X_o$ such that $\xi$ restricted to $Y_\star$ is an isomorphism onto $V_o$. 
\item [(ii)] For any irreducible component $Y \neq Y_\star$ of $X_o$, there exists $e \in E(\cc)$ such that $\xi(Y) = \{e\} \times F$ and the restriction $\xi|Y:Y \to F = \{e\} \times F$ is the fibration of either $T$ or $S$, so in particular $Y \in \{T,S\}$.
\item [(iii)] If $Y \neq Y'$ are incident irreducible components, then $Y \cap Y'$ is a section of $\xi|Y$ with trivial normal bundle if $Y \neq Y_\star$, or a fiber of $V_o \to E$ if $Y=Y_\star$. 
\end{enumerate}
\end{defn}

Less opaquely, the third condition says that, if $Y \approx S$, $Y \cap Y'$ is the distinguished section, while if $Y \approx T$, then $Y \cap Y'$ is a fiber of the projection to ${\mathbb P}^1$. All curves on $Y_o$ which are either fibers of $Y_\star \cong E \times F \to E$ on $Y_\star$ or fibers of the projection to ${\mathbb P}^1$ of trivially ruled components will be called \emph{F-curves}. The F-curves along which $X_o$ is singular are called \emph{special F-curves}, while the others are called \emph{simple F-curves}. By abuse of notation, F-curves will typically be denoted by $F$ even if they are (isomorphic to, but) different from the usual curve $F$ in Situation \ref{sit: main situation}; we hope this will not cause any confusion. It is automatic from Definition \ref{defn: drawer} and Remark \ref{rmk: no complete curves on open Atiyah surface} that the components isomorphic to $S$ must be leaves of the dual tree.

\begin{defn}\label{defn: expansion}
An \emph{expansion} is a variety $X$ (assumed irreducible) together with a commutative diagram
$ \vcenter{ \hbox{
\begin{tikzpicture}[scale = 0.9]
\node (a1) at (-1,2) {$X$};
\node (a2) at (0,2) {$V$};
\node (b1) at (-1,1) {$\Delta'$};
\node (b2) at (0,1) {$\Delta$};

\draw [->, thin] (a1)--(a2);
\draw [->, thin] (b1)--(b2);
\draw [->, thin] (a1)--(b1);
\draw [->, thin] (a2)--(b2);
\end{tikzpicture} }}$
(from now on abbreviated as $X \to V_{\Delta'}$) such that:
\begin{enumerate}
\item [(i)] The map $\delta:(\Delta',o') \to (\Delta,o)$ is a geometric base change, cf. \S\ref{ssec: introduction to deformable maps section}.
\item [(ii)] The fiber $X_{o'} \to V_o$ of $X \to V$ over $o'$ is a drawer, cf. Definition \ref{defn: drawer}.
\item [(iii)] With notation analogous to that in Definition \ref{defn: drawer}, there exists an isomorphism (denoted by `$=$' below) which makes the diagram below commute
\begin{center}
\begin{tikzpicture}[scale = 0.8]
\node (a1) at (0,1.3)  {$X$};
\node (a2) at (4,1.3)  {$V_{\Delta'}$};
\node (b1) at (0,0)  {$X \backslash \overline{\left(X_{o'} \backslash Y_\star \right)}$};
\node (b2) at (4,0)  {$V_{\Delta'} \backslash \xi(X_{o'} \backslash Y_\star). $};
\draw [->, thin] (a1)--(a2);
\draw [double equal sign distance, thin] (b1)--(b2);
\draw [right hook->, thin] (b1)--(a1);
\draw [right hook->, thin] (b2)--(a2);
\end{tikzpicture}
\end{center}
\end{enumerate}
\end{defn}

Note that we do not require $X$ to be nonsingular in Definition \ref{defn: expansion}.

\begin{lem}
If $X \to V_{\Delta'}$ is an expansion and $F \subset X_{o'}$ is a simple F-curve, then ${\sh N}_{F/X}$ is either the trivial rank two bundle over $F$, or the rank two Atiyah bundle over $F$.
\end{lem}

\begin{proof}
Let $Y$ be the irreducible component of $X_{o'}$ which contains $F$, so either $Y \approx T$ or $Y = Y_\star$. We have the short exact sequence
$$ 0 \longrightarrow {\sh N}_{F/Y} \longrightarrow {\sh N}_{F/X} \longrightarrow {\sh N}_{Y/X}|_F \longrightarrow 0. $$
We claim that the two line bundles in the extension above are trivial. It is trivial that ${\sh N}_{F/Y} \cong {\sh O}_F$. Let $Y^c \subset X_{o'}$ be the union of all irreducible components of $X_{o'}$ other than $Y$. Note that ${\sh O}_X(Y)|_Y \cong {\sh O}_X(-Y^c)|_Y$ because ${\sh O}_X(Y) \otimes {\sh O}_X(Y^c) = {\sh O}_X(X_{o'})$ is the pullback of ${\sh O}_{\Delta'}(o')$, which is invertible on $\Delta'$. It follows that ${\sh N}_{Y/X}|_F = {\sh O}_Y(-Y \cap Y^c)|_F \cong {\sh O}_F$, by the third condition in Definition \ref{defn: drawer}. 
\end{proof}

We will carry out most the calculations in this section in the following setup. 

\begin{sit}\label{sit: flat family of primitive divisors}
Let $D \subset V$ be a divisor defining a flat family $\{ D_z \}_{z \in \Delta}$ of divisors over $\Delta$ such that each $D_z$ is algebraically equivalent to $\beta_z$.
\end{sit}

Recall that $D_z$ is reduced for $z \neq o$. 

\begin{lem}\label{lem: divisor class on F}
In Situation \ref{sit: flat family of primitive divisors}, let $X \to V_{\Delta'}$ be an expansion and $D'$ the closure of $(D \backslash D_o) \times_\Delta \Delta'$ inside $X$. Assume that $D'$ is a Cartier divisor on $X$. Then ${\sh O}_X(D')|_F = {\sh O}_F(y)$, for any F-curve $F \subset X_{o'}$.
\end{lem}

\begin{proof}
The set of F-curves with image $\xi(F)$ in $V_o$ is naturally in bijection with the set of $\cc$-points of a rational tree $Q$. Consider the morphism $Q \to \mathrm{Pic}(F)$ given by $[F'] \mapsto {\sh O}_X(D')|_{F'}$ on $\cc$-points. Since $\mathrm{Pic}(F)$ doesn't contain any rational curves, we conclude that the map is constant and the lemma follows since it trivially holds in the special case $F \subset V_\star$.
\end{proof}

\begin{rmk}\label{rmk: what lemma "divisor class on F" means}
Lemma \ref{lem: divisor class on F} and the rest of the discussion so far give a clear picture of what the central fiber of $D'$ can look like. Let $Y$ be an irreducible component of $X_{o'}$ and let $D'|Y$ be the restriction of $D'$ to $Y$.
\begin{enumerate}
\item if $Y = Y_\star$, then $D'|Y$ is the sum of one fiber of $Y_\star \to F$ and several fibers of $Y_\star \to E$;
\item if $Y \approx T$, then $D'|Y$ consists of one rational fiber of $T$ and several F-curves;
\item if $Y \approx S$, then $D'|Y$ is the sum of a multiple of the distinguished section and an integral curve intersecting the distinguished section at the point which is mapped to $y$ \cite[\S2]{[Za19a]}.
\end{enumerate}
\end{rmk}

We now discuss the semistable reduction process: we want to make $D_o$ reduced (without allowing it to contain any special F-curve) by a sequence of blowups and base changes.

\begin{lem}\label{lem: bubbling up step}
With the setup in Lemma \ref{lem: divisor class on F}, assume that $\mathrm{Supp}(D'_{o'})$ doesn't contain any special F-curve. If $D'_{o'}$ is nonreduced, there exists a diagram
\begin{center}
\begin{tikzpicture}
\node (a0) at (-4,2) {$\widetilde{X}$};
\node (a1) at (-2,2) {$X$};
\node (a2) at (0,2) {$V$};
\node (b0) at (-4,1) {$\Delta''$};
\node (b1) at (-2,1) {$\Delta'$};
\node (b2) at (0,1) {$\Delta$};

\draw [->, thin] (a0)--(a1);
\draw [->, thin] (a1)--(a2);
\draw [->, thin] (b0)--(b1);
\draw [->, thin] (b1)--(b2);
\draw [->, thin] (a1)--(b1);
\draw [->, thin] (a2)--(b2);
\draw [->, thin] (a0)--(b0);
\end{tikzpicture}
\end{center}
such that the convex hull is an expansion, and, if $D''$ is the closure of $(D' \backslash D'_{o'}) \times_{\Delta'} \Delta''$ in $\smash{ \widetilde{X} }$, $D''$ is still Cartier, $\mathrm{Supp}(D''_{o''})$ also doesn't contain any special F-curve and $D''_{o''}$ is `less non-reduced' than $D'_{o'}$, e.g. in the sense that the sum of one less the multiplicities of its components is less than that for $D'_{o'}$. \end{lem}

\begin{proof}
This follows from Lemma \ref{lem: eventually lemma} in \S\ref{ssec: algebraic bubbling up}. By Remark \ref{rmk: what lemma "divisor class on F" means}, the nonreduced components of $D'_{o'}$ can only be F-curves. Let $F \subset D'_{o'}$ be an F-curve appearing with multiplicity $m \geq 2$. Let $e \in F$ be a closed point which doesn't belong to the other component of $D'_{o'}$ intersecting $F$. Choose formal coordinates $\smash{\widehat{\sh O}_{X,e} \cong \cc [\![ x,y,t ]\!]}$ on $X$ at $e$ such that $t$ is the pullback of a uniformizer of $\smash{ \widehat{\sh O}_{\Delta',o'} }$, which implies that $X_{o'}$ is given locally by $t=0$, and inside $X_{o'}$, $y$ is the pullback of a uniformizer of $\smash{\widehat{\sh O}_{\xi(F),\xi(e)} }$, $F$ is given by $x=0$ and thus $D'_{o'}$ is given by $x^m = 0$. Let $\smash{ \Phi \in \widehat{\sh O}_{X,e} }$ cut out $D'$ formally locally at $e$. By the formal Weierstrass Preparation Theorem, there exists a unit $\smash{u \in \widehat{\sh O}_{X,e} }$ and a Weierstrass polynomial 
$$ p(t,x,y) = x^m + \sum_{j=1}^m \alpha_j(t,y) x^{m-j}, \quad \alpha_j(t,y) \in (t,y) \subset \cc [\![ y,t ]\!], $$ 
such that $\Phi = up$. However, plugging in $t=0$ and recalling that the formal local equation of $D'_{o'}$ inside $X_{o'}$ is $x^m = 0$, we conclude that in fact $\alpha_j(t,y) \in (t)$. In other words, in the notation of \S\ref{ssec: algebraic bubbling up}, we have $p \in W_m(\cc [\![ y ]\!])$. The variables $t,x$ have the same meaning as in \S\ref{ssec: algebraic bubbling up}. Since $u$ is a unit, $D'$ is also locally cut out by $p$. Note that $p$ is not a perfect $m$th power since $D' \backslash D'_{o'}$ is reduced.

Let $\mu$ be the positive integer provided by Lemma \ref{lem: eventually lemma} for $p \in W_m(\cc [\![ y ]\!])$ and let $(\Delta'',o'') \to (\Delta',o')$ be a base change of local degree $\mu$ at $o'$. Let $p_0,p_1, \ldots ,p_N \in W_m(\cc [\![ y ]\!])$ such that 
\begin{enumerate} 
\item [$\bullet$] $p_0(t,x,y) = p(t^\mu,x,y)$;
\item [$\bullet$] $t^{-m}p_n(t,tx,y) \in \cc [\![ y,t]\!] [x]$ for $n=0,1,\ldots, N$;
\item [$\bullet$] $p_{n+1}(t,x,y) = t^{-m}p_n(t,t(x + \psi_n(y)),y)$ for some $\psi_n \in \cc [\![ y ]\!]$; and 
\item [$\bullet$] $t^{-m}p_N(t,t(x + \psi(y)),y) \notin W_m(\cc [\![ y ]\!])$ for all $\psi \in \cc [\![ y ]\!]$;
\end{enumerate}
that is, the sequence of odd rank iterates considered in Lemma \ref{lem: eventually lemma}. Let $X_0 = X \times_{\Delta'} \Delta''$ and $D_0 = D' \times_{\Delta'} \Delta''$. Note that $D_0 \subset X_0$ is still Cartier. Choose formal coordinates $x_0,y_0,t_0$ at the set-theoretic preimage $e_0$ of $e$ in $X_0$ such that $x \mapsto x_0$, $y \mapsto y_0$ and $t \mapsto t_0^\mu$ under the pullback $\smash{ \widehat{\sh O}_{X,e} \to \widehat{\sh O}_{X_0,e_0} }$. The formal equation of $D_0$ at $e_0$ is $p_0(t_0,x_0,y_0) = 0$. Let $F_0 = (F \times_{\Delta'} \Delta'') \cap (X_0 \times_{\Delta''} \{o''\})$. We blow up $F_0$ in $X_0$ and obtain $X_1$. Let $D_1$ be the proper transform of $D_0$ and $e_1 \in D_1$ the point mapping to $e_0 \in D_0$. If the component of $D_1$ containing $e_1$ is not an F-curve of $X_1$, we stop here. If it is, then, in the formal coordinates $x_1,y_1,t_1$ on $X_1$ at $e_1$ such that $x_0 \mapsto t_1(x_1 + \psi_0(y_1))$, $y_0 \mapsto y_1$, $t_0 \mapsto t_1$ under the pullback $\smash{ \widehat{\sh O}_{e_0,X_0} \to \widehat{\sh O}_{e_1,X_1}}$ (obviously, the image of $x_0$ is a multiple of $t_1$), $D_1$ is described by the equation $p_1(t_1,x_1,y_1) = 0$. Let $F_1 \subset X_1$ be this F-curve. We continue inductively.
\begin{center}
\begin{tikzpicture}[scale = 0.9]
\node (a2) at (2,1) {$X_M$};
\node (a3) at (4,1) {$\cdots$};
\node (a4) at (6,1) {$X_1$};
\node (a5) at (8,1) {$X_0$};
\node (a6) at (10,1) {$X \times_{\Delta'} \Delta''$};
\node (a7) at (12,1) {$X$};
\node (a8) at (14,1) {$S$};
\node (b2) at (2,0) {$\Delta''$};
\node (b3) at (4,0) {$\cdots$};
\node (b4) at (6,0) {$\Delta''$};
\node (b5) at (8,0) {$\Delta''$};
\node (b6) at (10,0) {$\Delta''$};
\node (b7) at (12,0) {$\Delta'$};
\node (b8) at (14,0) {$\Delta$};

\draw [->, thin] (a2)--(a3);
\draw [->, thin] (a3)--(a4);
\draw [->, thin] (a4)--(a5);
\draw [double equal sign distance, thin] (a5)--(a6);
\draw [->, thin] (a6)--(a7);
\draw [->, thin] (a7)--(a8);

\draw [double equal sign distance, thin] (b2)--(b3);
\draw [double equal sign distance, thin] (b3)--(b4);
\draw [double equal sign distance, thin] (b4)--(b5);
\draw [double equal sign distance, thin] (b5)--(b6);
\draw [->, thin] (b6)--(b7);
\draw [->, thin] (b7)--(b8);

\draw [->, thin] (a2)--(b2);
\draw [->, thin] (a4)--(b4);
\draw [->, thin] (a5)--(b5);
\draw [->, thin] (a6)--(b6);
\draw [->, thin] (a7)--(b7);
\draw [->, thin] (a8)--(b8);

\end{tikzpicture}
\end{center}
The fact that we never hit the [lose] branch of $\nu$ in Lemma \ref{lem: eventually lemma} ensures that all $F_i$ are simple F-curves. Lemma \ref{lem: eventually lemma} ensures that the process will end after at most $N$ steps. The last $D_M \subset X_M$ doesn't contain any special F-curves and its fiber over $o''$ is `less non-reduced' than $D'_{o'}$.
\end{proof}

Applying Lemma \ref{lem: bubbling up step} repeatedly, we conclude the following.
  
\begin{cor}\label{cor: expansions conclusion}
In Situation \ref{sit: flat family of primitive divisors}, there exists an expansion $X \to V_{\Delta'}$ such that if $D'$ is the closure of $(D \backslash D_o) \times_\Delta \Delta'$ in $X$, then $D'_{o'}$ is reduced and doesn't contain any special F-curves.
\end{cor}

\begin{rmk}\label{rmk: addition to conclusion of expansions}
By Remark \ref{rmk: what lemma "divisor class on F" means}, in the situation of Corollary \ref{cor: expansions conclusion}, any irreducible component of $D'_{o'}$ must be one of the following:
\begin{enumerate}
\item [(a)] a simple F-curve;
\item [(b)] a fiber $E \times \{pt\} \subset Y_\star \cong E \times F$;
\item [(c)] a rational fiber of some component $Y \approx T$ of $X_{o'}$; or
\item [(d)] a curve in a component $Y \approx S$ of $X_{o'}$ which intersects the distinguished section at only one point scheme-theoretically.
\end{enumerate}
Moreover, each irreducible component of $X_{o'}$ contains precisely one component of $D'_{o'}$ of one of the types (b)--(d) above and the components of type (c) are contracted by $X_{o'} \to V_o = E \times F$.
\end{rmk}

\begin{tiny}
\begin{center}
\begin{tikzpicture}

\def\kk{0.2}

\def\xx{2}
\def\yy{1}

\def\aa{0}
\def\bb{0}

\draw[thick] (0,2) to (\xx,\yy+2);

\draw[thick] (0,2.3) to (\xx,\yy+2.3);

\draw [thick, draw=black, fill=gray, opacity=0.1]
(\aa,\bb) to (\aa,\bb+5) to (\aa+\xx,\bb+5+\yy) to (\aa+\xx,\bb+\yy) to cycle;

\draw [black, thick] (\aa+\kk*\xx,\bb+\kk*\yy) to (\aa+\kk*\xx,\bb+\kk*\yy + 5);

\draw [thick, draw=black, fill=gray, opacity=0.1]
(\aa,\bb+1.2) to (\aa+\xx,\bb+\yy + 1.2) to (\aa+\xx + 3,\bb+\yy + 1.2) to (\aa+ 3,\bb+ 1.2) to cycle;

\draw [black, thick] (\aa+2.3,\bb+1.2 + 1.2) to (\aa+2.3+\xx,\bb+\yy + 1.2 + 1.2);

\draw [black, thick] (\aa+2.5,\bb+1.2 + 1.2) to (\aa+2.5+\xx,\bb+\yy + 1.2 + 1.2);

\draw [black, thick] (\aa+2.7,\bb+1.2 + 1.2) to (\aa+2.7+\xx,\bb+\yy + 1.2 + 1.2);

\draw [black, thick] (\aa+2.3,\bb+1.2) to (\aa+2.3+\xx,\bb+\yy + 1.2);

\draw [black, thick] (\aa+\kk*\xx,\bb+\kk*\yy + 1.2) to (\aa+\kk*\xx + 3,\bb+\kk*\yy + 1.2);

\draw [thick, draw=black, fill=gray, opacity=0.1]
(\aa+1.8,\bb+1.2) to (\aa+1.8+\xx,\bb+\yy + 1.2) to (\aa+1.8+\xx + 0.0, \bb+\yy + 1.2 + 1.5) to (\aa+1.8 + 0.0, \bb + 1.2 + 1.5) to cycle;

\draw [black, thick] (\aa+1.8+\kk*\xx,\bb+\kk*\yy + 1.2) to (\aa+1.8+\kk*\xx + 0.0,\bb+\kk*\yy + 1.2 + 1.5);

\draw [thick, draw=black, fill=gray, opacity=0.1]
(\aa + 1.8 + 0.0  ,\bb+1.2 + 1.2) to (\aa+1.8+\xx + 0.0,\bb+\yy + 1.2 + 1.2) to (\aa+1.8+\xx + 0.0 + 3,\bb+\yy + 1.2 + 1.2) to (\aa+1.8+ 0.0 + 3,\bb + 1.2 + 1.2) to cycle ;

\draw [black,  thick] (2.2, 2.6) to (2.2 + 3, 2.6);

\draw [thick, draw=black, fill=gray, opacity=0.1]
(4,2.4) to [bend right] (4+0.15*\xx,2.4 + 0.15 * \yy + 1.7) to (4 + 0.85* \xx,2.4 + 0.85 * \yy + 1.7) to [bend right] (4 + \xx,2.4 + \yy) to cycle;

\draw [black, thick] (4.4, 2.6) to[out=60,in=270] (5.4, 4.1);
\draw [black, thick] (5.4, 4.1) to[out=90, in=0] (5,4.5);
\draw [black, thick] (5,4.5) to[out=180, in=180] (5.6,3.3);

\draw [thick, draw=black, fill=gray, opacity=0.1]
(\aa+2.6,\bb+1.2) to (\aa+2.6+\xx,\bb+\yy + 1.2) to[bend right] (\aa+2.6+0.85*\xx , \bb+0.85*\yy + 1.2 - 1.5) to (\aa+ 0.15*\xx+2.6 , \bb + 0.15*\yy+ 1.2 - 1.5) to[bend right] cycle;

\draw [black, thick] (3,1.4) to[out = 270, in = 180] (3.5,0.4);
\draw [black, thick] (3.5,0.4) to [out = 0, in = 300] (4.1, 1.2);
\draw [black, thick] (4.1, 1.2) to [out = 120, in = 90] (3.1,0.1);

\draw [thick, draw=black, fill=gray, opacity=0.1]
(\aa,\bb+4.2) to (\aa+\xx,\bb+\yy + 4.2) to (\aa+\xx + 3,\bb+\yy + 4.2) to (\aa+ 3,\bb+ 4.2) to cycle;

\draw [black, thick] (\aa+\kk*\xx,\bb+\kk*\yy + 4.2) to (\aa+\kk*\xx + 3,\bb+\kk*\yy + 4.2);

\draw [black, thick] (\aa+1,\bb+4.2) to (\aa+1+\xx,\bb+\yy + 4.2);

\draw [black, thick] (\aa+2.3,\bb+4.2) to (\aa+2.3+\xx,\bb+\yy + 4.2);

\node at (3.8,0) {$S$};
\node at (5.9,4.1) {$S$};
\node at (0.8,5.7) {$Y_\star$};
\node at (3.4,5.4) {$T$};
\end{tikzpicture}

\textbf{Figure A.} An illustration of Corollary \ref{cor: expansions conclusion} and Remark \ref{rmk: addition to conclusion of expansions}.
\end{center}
\end{tiny}

\begin{proof}[Proof of Theorem \ref{thm: necessary condition to deform}.] 
Making the base change in Definition \ref{defn: deformable}, we may assume without loss of generality that there exists a stable map $[f] \in \overline{\modu M}_g(V/\Delta,\beta)(\Delta)$ such that $f_o \approx h$. Let $D \subset V$ be the image of $f$. By Corollary \ref{cor: expansions conclusion}, there exists a base change $(\Delta',o') \to (\Delta,o)$ and an expansion $X \to V_{\Delta'}$ such that, if $D'$ is the closure of $(D \backslash D_o) \times_\Delta \Delta'$ in $X$, then $D'_{o'}$ is reduced and doesn't contain any special F-curves. The restriction of $(\Delta' \to \Delta)^* f$ to $\Delta' \backslash \{o'\}$ is a stable map into $V_{\Delta' \backslash \{o'\}} = X_{\Delta' \backslash \{o'\}}$, so, by nodal reduction/properness of the moduli spaces, there exists another base change $(\Delta'',o'') \to (\Delta',o')$ and a map $f''$ from some prestable curve over $\Delta''$ into $X_{\Delta''}$ whose restriction to $\Delta''\backslash \{o''\}$ is precisely $(\Delta''\backslash \{o''\} \to \Delta)^*f$. By the separateness of the moduli spaces of stable maps, the prestable map $(X_{\Delta''} \to V_{\Delta''}) \circ f''$ `stabilizes' to $(\Delta'' \to \Delta)^*f$ (to stabilize, contract the destabilizing components and use \cite[Lemma 2.2]{[BM96]}). In particular, $(X_{o'} \to V_o) \circ f''_{o''}$ stabilizes to $h$. 

All irreducible components of the source of $f''_{o''}$ are of one of five possible types:
\begin{enumerate}
\item [(a)--(d)] corresponding to the types (a)--(d) in Remark \ref{rmk: addition to conclusion of expansions}; and possibly
\item [(e)] contracted components.
\end{enumerate}
The images in $X_{o'}$ of the components of types (a)--(d) are all different, and all type (e) components are rational by assumption. 

Let $C$ be an irreducible component of the source of $h$, and $C''$ the corresponding irreducible component of the source of $f''_{o''}$. Assume that this is a positive genus component, and not the unique type (b) component. It is not of type (c) because it's irrational, not of type $(e)$ because $h$ doesn't contract $C$ and also by definition not of type (b), so it is of type (a) or (d). The latter claim, i.e. the fact that the restriction of the map is quasi-traceless, now follows automatically from Lemma \ref{lem: connection to Atiyah surface} and the trivial fact that isomorphisms are always quasi-traceless. Finally, we show that $C$ corresponds to a leaf of the dual graph of $h$. Obviously, the source of $f''_{o''}$ is of compact type, because the source of $h$ is of compact type. The desired conclusion trivially follows from the next claim: the configuration of descendants of $C''$ in the dual graph is a forest of rational components, which will be contracted to a finite set of points by $(X_{o'} \to V_o) \circ f''_{o''}$. First, there can't be any descendants of positive genus (if $C''$ is of type (d), it is useful to keep in mind Remarks \ref{rmk: addition to conclusion of expansions}, or \ref{rmk: what lemma "divisor class on F" means}, or \ref{rmk: no complete curves on open Atiyah surface} to see this). Second, the rational descendants are trivially contracted by $(X_{o'} \to V_o) \circ f''_{o''}$ since all maps from ${\mathbb P}^1$ into a positive genus curve are constant. Finally, the configuration of descendants is a forest regardless because the source of $f''_{o''}$ is of compact type, and the proof of Theorem \ref{thm: necessary condition to deform} is complete. 
\end{proof}

We are inclined to believe that Theorem \ref{thm: necessary condition to deform} is a `shadow' of a somewhat different phenomenon. If true, this would give a better picture of how the deformable maps vary in moduli.

\begin{Q}\label{Q: main question for future}
In the current situation, does there exist an alternate degeneration/compactification of the moduli space of stable maps (aesthetically) in the style of \cite{[Li01], [KKO14]}? Does a similar construction exist for K3 surfaces?
\end{Q}

\subsection{Proof of Claim \ref{claim: key claim}}\label{ssec: proof of main claim} Let $\varpi = (\Lambda_1 + \cdots + \Lambda_{g-1} = \Lambda[\theta]) \in \Pi$, and $\varpi \succ \tilde{\varpi}$ obtained by replacing $\Lambda_i$ and $\Lambda_j$ with $\Lambda_i+\Lambda_j$, with notation as in \S\ref{ss: reduction ss} and \S\ref{ssec: partitions combinatorics}. With Remark \ref{rmk: factorizations on product surfaces} in mind, we introduce the following class of stable maps. 

\begin{defn}\label{defn: qs and qsqt}
A stable map $[f] \in {\modu T}_o(\cc)$ is \emph{quasi-simple of type $\varpi \succ \tilde{\varpi}$} if  
\begin{enumerate}
\item the source of $f$ is $\smash{B \cup T_{ij} \cup \bigcup_{k \in [g-1] \backslash\{i,j\}} T_k }$ such that $T_{ij}$ is smooth of genus $2$, while $B$ and $T_k$ are smooth of genus $1$, and the dual graph is a star centered at $B$; 
\item $f|B$ is an isomorphism onto a fiber of $V_o \to F$, $f|T_k$ is an isogeny onto a fiber of $V_o \to E$ with associated lattice $\Lambda_k$, and $f|T_{ij}$ is a (ramified) cover of a fiber of $V_o \to E$, of degree $d_i+d_j$ and with associated lattice $\Lambda_i+\Lambda_j$, that is, the image of $\hh_1(f|T_{ij})$ is $\Lambda_i+\Lambda_j \subseteq \hh_1(F,\zz)$. 
\end{enumerate}
We say that $f$ is \emph{quasi-simple and quasi-traceless (qsqt) of type $\varpi \succ \tilde{\varpi}$} if $f|T_{ij}$ with the marking $T_{ij} \cap B \in T_{ij}$ is quasi-traceless onto $F$ in the sense of Definition \ref{defn: quasi-traceless}. 
\end{defn}

Again, the property of being quasi-simple generalizes easily first to algebraically closed extensions of $\cc$, and then to families by imposing the condition on geometric fibers. We leave it to the reader to check that the category ${\modu T}[\varpi \succ \tilde{\varpi}]$ of quasi-simple maps of type $\varpi \succ \tilde{\varpi}$ is naturally an open substack of ${\modu T}_o$, and that there is a natural $1$-morphism 
\begin{equation}\label{eqn: split off the genus 2} {\modu T}[\varpi \succ \tilde{\varpi}] \longrightarrow {\modu M}^\mathrm{sm}_{2,1}(F,d_i+d_j;\Lambda_i+\Lambda_j) \end{equation}
which on $\cc$-points associates $(T_{ij},f|T_{ij},T_{ij} \cap B)$ to $[f]$ with notation as in Definition \ref{defn: qs and qsqt}, where ${\modu M}^\mathrm{sm}_{2,1}(F,d_i+d_j;\Lambda_i+\Lambda_j) \subset \overline{\modu M}_{2,1}(F,d_i+d_j;\Lambda_i+\Lambda_j)$ is the open substack of stable maps with smooth sources. The arguments are similar to those for ${\modu V}[\varpi]$, which were sketched in \S\ref{ss: reduction ss}.

We define the moduli stack ${\modu Q}[\varpi \succ \tilde{\varpi}]$ of qsqt type $\varpi \succ \tilde{\varpi}$ stable maps by constructing the cartesian diagram
\begin{center}
\begin{tikzpicture}
\matrix [column sep  = 20mm, row sep = 5mm] {
	\node (a1) {${\modu Q}[\varpi \succ \tilde{\varpi}]$}; &
	\node (a2) {${\modu T}[\varpi \succ \tilde{\varpi}]$}; \\
	\node (b1) {${\modu Q}^\mathrm{sm}_{2,1}(F,d_i+d_j;\Lambda_i+\Lambda_j)$}; &
	\node (b2) {${\modu M}^\mathrm{sm}_{2,1}(F,d_i+d_j;\Lambda_i+\Lambda_j)$,}; \\
};
\draw[right hook->, thin] (a1) -- (a2);
\draw[right hook->, thin] (b1) -- (b2);
\draw[->,thin] (a1) -- (b1);
\draw[->,thin] (a2) -- (b2);
\node at (3.4,0) {\eqref{eqn: split off the genus 2}};
\node at (0,-0.3) {\eqref{eqn: Q closed immersion}, \eqref{eqn: Q decomp}};
\end{tikzpicture}
\end{center}
where ${\modu Q}^\mathrm{sm}_{2,1}(F,d_i+d_j;\Lambda_i+\Lambda_j) \subset {\modu M}^\mathrm{sm}_{2,1}(F,d_i+d_j;\Lambda_i+\Lambda_j)$ is the moduli space of quasi-traceless covers, analogously to \eqref{eqn: Q closed immersion}. Note that \eqref{eqn: split off the genus 2} has irreducible geometric fibers of dimension $g-2$, hence so does ${\modu Q}[\varpi \succ \tilde{\varpi}] \to {\modu Q}^\mathrm{sm}_{2,1}(F,d_i+d_j;\Lambda_i+\Lambda_j)$. By Theorem \ref{thm: quasi-traceless final theorem}, ${\modu Q}[\varpi \succ \tilde{\varpi}]$ is irreducible of dimension $g$.

\begin{rmk}\label{rmk: check qsqt at complex points}
If $N$ is reduced and of finite type over $\cc$, and $[f] \in {\modu T}[\varpi \succ \tilde{\varpi}](N)$ is a stable map over $N$, then, if we wish to show that $[f] \in {\modu Q}[\varpi \succ \tilde{\varpi}](N)$, it suffices to check that $[f_p] \in {\modu Q}[\varpi \succ \tilde{\varpi}](\cc)$ for all $p \in N(\cc)$, cf. Remark \ref{rmk: testing quasi-traceless on complex points}. 
\end{rmk}

We introduce one last type of stable map, again, keeping Remark \ref{rmk: factorizations on product surfaces} in mind. The essential steps in the proof of Claim \ref{claim: key claim} will operate locally near such a stable map.  

\begin{defn}\label{defn: transitional map}
A stable map $[f] \in {\modu T}_o(\cc)$ is \emph{transitional of type $\varpi \succ \tilde{\varpi}$} if  
\begin{enumerate}
\item the source of $f$ is $B \cup G \cup T_1 \cup \cdots \cup T_{g-1}$, such that $G$ is smooth of genus $0$, all other components are smooth of genus $1$, and the dual graph is
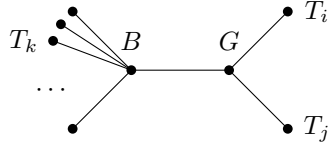
\begin{center}
\begin{tikzpicture}[scale = 1.3]
\fill (0,0) circle (0.5mm);
\fill (1,0) circle (0.5mm);
\fill (1.6,0.6) circle (0.5mm);
\fill (1.6,-0.6) circle (0.5mm);
\fill (-0.6,0.6) circle (0.5mm);
\fill (-0.8,0.3) circle (0.5mm);
\fill (-0.72,0.47) circle (0.5mm);
\fill (-0.6,-0.6) circle (0.5mm);
\draw (0,0) -- (1,0);
\draw (1,0) -- (1.6,0.6);
\draw (1,0) -- (1.6,-0.6);
\draw (0,0) -- (-0.6,0.6);
\draw (0,0) -- (-0.8,0.3);
\draw (0,0) -- (-0.72,0.47);
\draw (0,0) -- (-0.6,-0.6);

\node at (-1.1,0.3) {$T_k$}; 
\node at (1.9,0.6) {$T_i$}; 
\node at (1.9,-0.6) {$T_j$};
\node at (-0.8,-0.2) {$\cdots$};
\node at (0,0.3) {$B$};
\node at (1,0.3) {$G$};
\end{tikzpicture}
\end{center}
\item $G$ is a contracted (ghost) component, $f|B$ is an isomorphism onto a fiber of $V_0 \to F$, and $f|T_k$ is an isogeny onto a fiber of $V_o \to E$ with associated lattice $\Lambda_k$, for all $k \leq g-1$.
\end{enumerate}
\end{defn}

Fix $y \in {\modu T}_o(\cc)$ corresponding to a transitional $\varpi \succ \tilde{\varpi}$-map. 

\begin{tiny}
\begin{center}
\begin{tikzpicture}[scale = 1.1]

\def\aa{0}

\draw [thick, draw=black, fill=gray, opacity=0.1]
	(-1.4+\aa,-0.7) to (-1.4+\aa,0.7) to (1.4+\aa,0.7) to (1.4+\aa,-0.7) to cycle;
	
\draw (-1.4+\aa,-0.3) to (1.4+\aa,-0.3);

\draw (-1+\aa,-0.7) to (-1+\aa,0.7);
\draw (0.7+\aa,-0.7) to (0.7+\aa,0.7);
\draw [thick, very thick] (0.2+\aa,-0.7) to (0.2+\aa,0.7);

\def\hh{2.5}
\def\d{1}
\def\dd{-0.3}

\draw (\aa-\d,-1+\hh) to [bend left] (\aa-\d,0+\hh) to [bend right] (\aa-\d,1+\hh);

\draw (\aa - \d + \dd, -0.7+\hh) to [in=170, out=10] (\aa-\d + \dd + 2.5, -0.7+\hh);

\draw [thick] (\aa - \d + \dd, -0.1+\hh) to [in=170, out=10] (\aa-\d + \dd + 2.5, -0.1+\hh);

\draw (\aa - \d + \dd, 0.7+\hh) to [in=170, out=10] (\aa-\d + \dd + 2.5, 0.7+\hh);

\draw [thick] (\aa + 0.6, -0.4+\hh) to [in=-100, out = 100] (\aa + 0.6, 0.6+\hh);
\draw [thick] (\aa -0.2, -0.4+\hh) to [in=-100, out = 100] (\aa -0.2, 0.6+\hh);

\draw [->] (\aa, 1.7) to (\aa, 0.85);

\def\aa{-4}

\draw [thick, draw=black, fill=gray, opacity=0.1]
	(-1.4+\aa,-0.7) to (-1.4+\aa,0.7) to (1.4+\aa,0.7) to (1.4+\aa,-0.7) to cycle;
	
\draw (-1.4+\aa,-0.3) to (1.4+\aa,-0.3);

\draw (-1+\aa,-0.7) to (-1+\aa,0.7);
\draw (0.7+\aa,-0.7) to (0.7+\aa,0.7);
\draw [thick, very thick] (0.2+\aa,-0.7) to (0.2+\aa,0.7);

\def\hh{2.5}
\def\d{1}
\def\dd{-0.3}

\draw (\aa-\d,-1+\hh) to [bend left] (\aa-\d,0+\hh) to [bend right] (\aa-\d,1+\hh);

\draw (\aa - \d + \dd, -0.7+\hh) to [in=170, out=10] (\aa-\d + \dd + 2.5, -0.7+\hh);
red
\draw [thick] (\aa - \d + \dd, -0.1+\hh) to [in=170, out=10] (\aa-\d + \dd + 2.5, -0.1+\hh);

\draw (\aa - \d + \dd, 0.7+\hh) to [in=170, out=10] (\aa-\d + \dd + 2.5, 0.7+\hh);

\draw [->] (\aa, 1.7) to (\aa, 0.85);

\def\aa{4}

\draw [thick, draw=black, fill=gray, opacity=0.1]
	(-1.4+\aa,-0.7) to (-1.4+\aa,0.7) to (1.4+\aa,0.7) to (1.4+\aa,-0.7) to cycle;
	
\draw (-1.4+\aa,-0.3) to (1.4+\aa,-0.3);

\draw (-1+\aa,-0.7) to (-1+\aa,0.7);
\draw (0.7+\aa,-0.7) to (0.7+\aa,0.7);
\draw [thick, very thick] (0.3+\aa,-0.7) to (0.3+\aa,0.7);
\draw [thick, very thick] (0.1+\aa,-0.7) to (0.1+\aa,0.7);

\def\hh{2.5}
\def\d{1}
\def\dd{-0.3}

\draw (\aa-\d,-1+\hh) to [bend left] (\aa-\d,0+\hh) to [bend right] (\aa-\d,1+\hh);

\draw (\aa - \d + \dd, -0.7+\hh) to [in=170, out=10] (\aa-\d + \dd + 2.5, -0.7+\hh);

\draw [thick] (\aa - \d + \dd, -0+\hh) to [in=170, out=10] (\aa-\d + \dd + 2.5, -0+\hh);

\draw [thick] (\aa - \d + \dd, -0.2+\hh) to [in=170, out=10] (\aa-\d + \dd + 2.5, -0.2+\hh);

\draw (\aa - \d + \dd, 0.7+\hh) to [in=170, out=10] (\aa-\d + \dd + 2.5, 0.7+\hh);

\draw [->] (\aa, 1.7) to (\aa, 0.85);

\draw[->, decorate, decoration={zigzag,amplitude=0.9pt,segment length=1.2mm}] (-2.3,1.2) to (-1.7,1.2);
\draw[->, decorate, decoration={zigzag,amplitude=0.9pt,segment length=1.2mm}] (2.3,1.2) to (1.7,1.2);

\node at (4,3.7) {$\text{simple}$};
\node at (-4,3.7) {$\text{qsqt}$};
\node at (0,3.7) {$\text{transitional}$};
\node at (-3.3,2.7) {$\text{quasi-traceless}$};
\node at (4.7,2.8) {$\text{isogenies}$};
\node at (1.1,2.6) {$\text{ghost}$};

\end{tikzpicture}

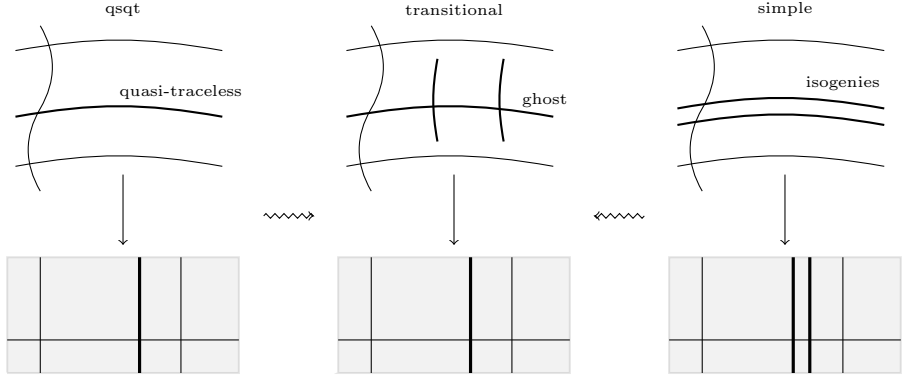
\bigskip \textbf{Figure B.} An illustration of Lemma \ref{lem: figure B lemma}.
\end{center}
\end{tiny}

\begin{lem}\label{lem: figure B lemma}
The point $y$ lies in the closures of both $|{\modu Q}[\varpi \succ \tilde{\varpi}]|$ and $|{\modu V}[\varpi]|$ relative to $|{\modu T}_o|$ (and thus also relative to $|{\modu T}|$). 
\end{lem}

\begin{proof}
It follows from Theorem \ref{thm: quasi-traceless final theorem} that $y$ lies in the closure of $|{\modu Q}[\varpi \succ \tilde{\varpi}]|$. For the other claim, we can construct a specialization to $y$ completely explicitly. One such construction proceeds by blowing up $\{(t)\} \times \{\mathrm{pt}\} \times F$ in $\spec \cc[\![t]\!] \times E \times F$, defining a suitable family of maps into the blown up threefold, and then composing with the blowdown to obtain the desired specialization. For instance, the ghost component will correspond to a fiber of the (trivial) ruling of the exceptional divisor. The details are omitted.  
\end{proof}

The following proposition connects what happens above $o$ to what happens above $\Delta$. However, note that the only way in which $\Delta$ comes up is the requirement (please see below) that ${\modu S}$ dominates $\Delta$; without it, the proposition would be false. 

\begin{prop}\label{prop: key proposition}
Let ${\modu S}$ be a closed irreducible substack of ${\modu T}$ of dimension $g+1$, which dominates $\Delta$. If $y \in |{\modu S}|$, then $|{\modu Q}[\varpi \succ \tilde{\varpi}]| \subset |{\modu S}|$ and $|{\modu V}[\varpi]| \subset |{\modu S}|$.
\end{prop}

\begin{proof}
For the first part, the argument is informally the following: locally at $y$ in ${\modu S}$, the divisor where the node $B \cap G$ (notation as is Definition \ref{defn: transitional map}) `survives' must generically consist of maps in ${\modu Q}[\varpi \succ \tilde{\varpi}]$ by Theorem \ref{thm: necessary condition to deform}, so it must contain the generic point of $|{\modu Q}[\varpi \succ \tilde{\varpi}]|$ for obvious dimension reasons. Let $\smash{ (U,u) \to ({\modu S},y) }$ be an elementary \'{e}tale neighborhood of $y$ in ${\modu S}$, which may shrink according to our needs. Let $f:C \to V_U$ be the corresponding stable map. 

Since $\overline{\modu M}_g \subset {\mathfrak M}_g$ is open, we can assume that the source $C$ is stable, so we have an induced $1$-morphism $U \to \overline{\modu M}_g$. By \cite[Corollary 3.9]{[Kn83]}, we may choose $(U,u)$ such that there exists a connected Cartier divisor $D \subset U$ which contains $u$, and a presentation of $C_D = D \times_U C$ as a clutching $C_D = Y_1 \cup_\rho Y_2$ along a section $\rho:D \to C_D$, such that $\rho(u)$ corresponds to the node $\smash{ B \cap G }$, and $Y_{2,u}$ contains the ghost component. Some care needs to be taken in the special case $g=4$ because of the hypothesis in \cite[Corollary 3.9.b)]{[Kn83]}, but we skip these details. 

Since $D \subset U$ is a Cartier divisor, it must have pure dimension $g$. We claim that $D^\mathrm{red} \subseteq U_o$. Indeed, for any $p \in D(\cc)$, we have 
$$ c_1({\sh L}_p) = f_{p,*}[Y_{1,p}] + f_{p,*}[Y_{2,p}] \in \mathrm{NS}(V_p), \quad \text{with} \quad f_{p,*}[Y_{1,p}], f_{p,*}[Y_{2,p}] > 0, $$
since $D$ is connected and degrees are locally constant, which, by the assumptions in \S\ref{ssec: Severi stacks}, implies that $p$ lies above $\partial M$, hence $p \in U_o$. Moreover, $f_p:C_p \to V_p = V_o$ must satisfy the properties in Theorem \ref{thm: necessary condition to deform}, for all $p \in D({\mathbb C})$, by Remark \ref{rmk: deformable alternative} and Theorem \ref{thm: necessary condition to deform}. We've boiled down the proposition to a technical claim which only `sees' $V_o$, which we state separately for the sake of clarity. 

\begin{claim}\label{claim: the technical claim that key proposition boils down to}
Let $N$ be a reduced, separated, finite type scheme over $\cc$ of pure dimension $g$, $p_0 \in N(\cc)$, and a $1$-morphism $\chi: N \to {\modu T}_o$ corresponding to a stable map $f:C \to V_o \times N$ over $N$, such that the following hold:
\begin{enumerate}
\item the morphism $\chi$ is quasi-finite, and $\chi(p_0) = y$;
\item the source $C$ is presented as a gluing $C= Y_1 \cup_\rho Y_2$ along a singular section $\rho:N \to C$, as above;
\item for any $p \in N(\cc)$, the stable map $\chi(p)$, that is, $f_p:C_p \to V_o$, satisfies the properties in the conclusion of Theorem \ref{thm: necessary condition to deform};
\item the source $C$ is stable (over $S$).
\end{enumerate}
Then any generic point $\nu \in N$ such that $p_0 \in \overline{\nu}$ is mapped by $\chi$ to the generic point of $|{\modu Q}[\varpi \succ \tilde{\varpi}]|$. 
\end{claim}

\begin{proof}
We begin with some technical reductions. The reader may skim over this paragraph and refer back as necessary. Note that $(Y_2,\rho)$ is stable marked curve, and of course $(Y_2,f|Y_2,\rho)$ is a stable map, so we have $1$-morphisms
$$ \widetilde{\chi}_2: N \longrightarrow \overline{\modu M}_{2,1} \quad \text{and} \quad \chi_2:N \longrightarrow \overline{\modu M}_{2,1}(F). $$
We may assume that $\chi_2$ factors through $\overline{\modu M}_{2,1}(F,d_i+d_j;\Lambda_i+\Lambda_j) \hookrightarrow \overline{\modu M}_{2,1}(F)$ by shrinking $N$. Let $e,f \in \hh_1(V_o,\zz)$ be the classes of the fibers of $V_o \to F$ and $V_o \to E$ respectively. We have
$$ (Y_1,f|Y_1,\rho) \in \overline{\modu M}_{g-2,1}(V_o, e +d'f)(N), \quad \text{where} \quad d'=d-d_i-d_j. $$ 
Let ${\modu M}^\mathrm{gen}_{g-2,1}(V_o, e +d'f)$ be the open substack parametrizing stable maps whose dual graph is of the `generic' type with the marking on the backbone; we may shrink $N$ so that $(Y_1,f|Y_1,\rho)$ belongs to this substack. Recall that clutching maps for moduli of stable maps can be constructed in general based on \cite[Proposition 2.4]{[BM96]}. The one relevant to our situation takes the form
\begin{equation}\label{eqn: unpleasant clutching map} \gamma:F \times \overline{\modu M}_{1,g-2}(E, 1) \longrightarrow \overline{\modu M}_{g-2,1}(V_o, e +d'f), \end{equation}
as explained below. (Coincidentally, $\overline{\modu M}_{1,g-2}(E, 1)$ is a Fulton-MacPherson configuration space \cite{[FM94]}, though we will ultimately only be concerned with its open stratum, the complement of the large diagonal in $\smash{ E^{g-2} }$.) Then $\gamma$ informally glues $g-3$ isogenies onto the fibers of $V_o \to E$ corresponding to $\Lambda_k$, $k \neq i,j$, to a (possibly bubbled-up, although this won't happen for the fibers of $f|Y_1$) fiber of $V_o \to F$ with $g-2$ markings, at $g-3$ of these markings. Note that $\gamma$ is finite, and in fact \'{e}tale above ${\modu M}^\mathrm{gen}_{g-2,1}(V_o, e +d'f)$, so, after passing to an elementary \'{e}tale neighborhood of $(N,p_0)$, we may assume that the $1$-morphism $\smash{ N \to \overline{\modu M}_{g-2,1}(V_o, e +d'f) }$ induced by $(Y_1,f|Y_1,\rho)$ also factors through $\gamma$.

Let $\Sigma^0$, $\Sigma^1$, $\Sigma^2$ be the strata of $\overline{\modu M}_{2,1}$ consisting of stable marked curves with the following dual graphs (the labels are the genera, and the leg is the marked point)
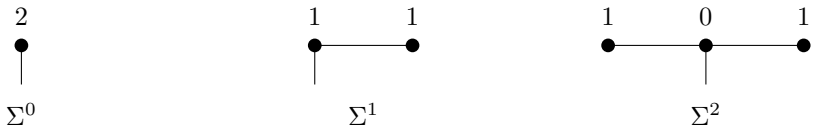
\begin{center}
\begin{tikzpicture}[scale = 1.3]
\fill (-1,0) circle (0.7mm);
\draw (-1,0) -- (-1,-0.4);
\node at (-1,-0.7) {$\Sigma^0$};
\node at (-1,0.3) {$2$};

\fill (2,0) circle (0.7mm);
\fill (3,0) circle (0.7mm);
\draw (2,0) -- (2,-0.4);
\draw (2,0) -- (3,0);
\node at (2.5,-0.7) {$\Sigma^1$};
\node at (2,0.3) {$1$};
\node at (3,0.3) {$1$};

\fill (5,0) circle (0.7mm);
\fill (6,0) circle (0.7mm);
\fill (7,0) circle (0.7mm);
\draw (6,0) -- (6,-0.4);
\draw (5,0) -- (6,0);
\draw (7,0) -- (6,0);
\node at (6,-0.7) {$\Sigma^2$};
\node at (5,0.3) {$1$};
\node at (6,0.3) {$0$};
\node at (7,0.3) {$1$};
\end{tikzpicture}
\end{center}
Trivially, $\Sigma^i$ has codimension $i$ in $\overline{\modu M}_{2,1}$. Let $\smash{ \Sigma^i_N = \widetilde{\chi}_2^{-1}(\Sigma^i) }$. Since $\Sigma^0 \cup \Sigma^1 \cup \Sigma^2$ is open in $\overline{\modu M}_{2,1}$, we may assume without loss of generality that 
\begin{equation}\label{eqn: unpleasant important claim decomposition}
N = \bigcup_{ i = 0,1, 2} \Sigma_i^N.
\end{equation}
The previous paragraph shows that the restriction of $f$ above $\Sigma_N^0$ is quasi-simple of type $\varpi \succ \varpi'$ (cf. Definition \ref{defn: qs and qsqt}). Rephrased,
\begin{equation}\label{eqn: unpleasant important claim factorization}
\text{the restriction of $\chi$ to $\Sigma_N^0$ factors through ${\modu T}[\varpi \succ \varpi'] \hookrightarrow {\modu T}_o$}. 
\end{equation}
The fibers of $\smash{\widetilde{\chi}_2}$ have dimension at most $g-3$, and hence
\begin{equation}\label{eqn: unpleasant important claim dimension count}
\dim \Sigma^2_N \leq \dim \Sigma^2 + g-3 = g-1.
\end{equation}
The third assumption and Theorem \ref{thm: necessary condition to deform} imply that
\begin{equation}\label{eqn: unpleasant important claim impossible configuration}
\Sigma^1_N = \emptyset.
\end{equation}
It follows from \eqref{eqn: unpleasant important claim decomposition}, \eqref{eqn: unpleasant important claim dimension count}, and \eqref{eqn: unpleasant important claim impossible configuration} that $\Sigma^0_N$ is dense in $N$. The third assumption in Claim \ref{claim: the technical claim that key proposition boils down to}, \eqref{eqn: unpleasant important claim factorization}, and Remark \ref{rmk: check qsqt at complex points} imply that the restriction of $f$ above $\Sigma_N^0$ is an object of ${\modu Q}[\varpi \succ \tilde{\varpi}]$, that is, the restriction of $\chi$ to $\Sigma_N^0$ even factors through the composition ${\modu Q}[\varpi \succ \varpi'] \hookrightarrow {\modu T}[\varpi \succ \varpi'] \hookrightarrow {\modu T}_o$. Then the first assumption in Claim \ref{claim: the technical claim that key proposition boils down to}, and the previously proved fact that $\Sigma^0_N$ is dense in $N$ complete the proof.
\end{proof}

Applying Claim \ref{claim: the technical claim that key proposition boils down to} with $N= D^\mathrm{red}$, we conclude that $|{\modu Q}[\varpi \succ \tilde{\varpi}]| \subset |{\modu S}|$. 

The idea to prove $|{\modu V}[\varpi]| \subset |{\modu S}|$ is similar: locally at $y$ in ${\modu S}$, the divisor where the node $G \cap T_i$ (we may equally well take $G \cap T_j$) survives must generically consist of maps in ${\modu V}[\varpi]$ by Theorem \ref{thm: necessary condition to deform}, so it must contain the generic point of $|{\modu V}[\varpi]|$ for trivial dimension reasons. The technical details are similar to the ones used to prove the first claim, and they are left to the reader. \end{proof}

Remark \ref{rmk: remark about closures}, Lemma \ref{lem: figure B lemma}, and Proposition \ref{prop: key proposition} imply the following.

\begin{cor}\label{cor: final corollary}
Let ${\modu Z}$ be an irreducible, or equivalently, connected, component of ${\modu V}$, and $\smash{ \overline{{\modu Z}}}$ its closure relative to ${\modu T}$ with the reduced structure. Then $|{\modu Z}|$ contains $|{\modu V}[\varpi]|$ if and only if $|\overline{\modu Z}|$ contains $|{\modu Q}[\varpi \succ \tilde{\varpi}]|$. 
\end{cor}

Claim \ref{claim: key claim} follows from Corollary \ref{cor: final corollary} since clearly ${\modu Q}[\varpi_1 \succ \varpi_3] = {\modu Q}[\varpi_2 \succ \varpi_3]$ for any roof (\S\ref{ssec: partitions combinatorics}) $\varpi_1 \succ \varpi_3 \prec \varpi_2$, so the proof of Theorem \ref{thm: main theorem} is complete.


\begin{thebibliography}{aaaaaaaa}

\bibitem[AM16]{[AM16]}
W. Alagala and A. Maciocia, \emph{Critical k-Very Ampleness for Abelian Surfaces}, Kyoto J. Math. \textbf{56}, 33--47 (2016)

\bibitem[ACG11]{[ACG11]} E. Arbarello, M. Cornalba, and Ph. Griffiths, \emph{Geometry of algebraic curves}, Vol. II, with a contribution by J. Harris, Grundlehren Math. Wiss. \textbf{268}, Springer (2011)

\bibitem[At57]{[At57]}
M. F. Atiyah, \emph{Vector bundles over an elliptic curve}, Proc. London Math. Soc. \textbf{7}, 414--452 (1957)

\bibitem[Ba19]{[Ba19]}
E. Ballico, \emph{On the irreducibility of the Severi variety of nodal curves in a smooth surface}, Arch. Math. \textbf{113}, 483--487 (2019)

\bibitem[Be97]{[Be97]}
K. Behrend, \emph{Gromov-Witten invariants in algebraic geometry}, Invent. Math. \textbf{127}, 601--617 (1997).

\bibitem[BM96]{[BM96]} 
K. Behrend and Yu. Manin, \emph{Stacks of stable maps and Gromov-Witten invariants}, Duke Math. J. \textbf{85}, 1--60 (1996)

\bibitem[BF97]{[BF97]}
K. Behrend and B. Fantechi, \emph{The intrinsic normal cone}, Invent. Math. \textbf{128}, 45--88 (1997)

\bibitem[BS91]{[BS91]}
M. Beltrametti and A. J. Sommese, \emph{Zero cycles and $k$-th order embeddings of smooth projective surfaces}, in \emph{Problems in the theory of surfaces and their classification} (Cortona, 1988), Sympos. Math., XXXII, 33--48. Academic Press (1991)

\bibitem[BL99]{[BL99]}
J. Bryan and C. Leung, \emph{Generating functions for the number of curves on abelian surfaces}, Duke Math. J. \textbf{99} (2), 311--328 (1999)

\bibitem[Bu15]{[Bu15]}
G. Bujokas, \emph{Covers of an elliptic curve $E$ and curves in $E \times {\mathbb P}^1$}, PhD Thesis, Harvard University (2015)

\bibitem[CD12]{[CD12]}
C. Ciliberto and Th. Dedieu, \emph{On universal Severi varieties of low genus K3 surfaces}, Math. Z. \textbf{271}, 953--960 (2012)

\bibitem[CD19]{[CD19]}
C. Ciliberto and Th. Dedieu, \emph{On the irreducibility of Severi varieties on K3 surfaces}, Proc. Amer. Math. Soc. \textbf{147}, 4233--4244  (2019)

\bibitem[CFGK17]{[CFGK17]}
C. Ciliberto, F. Flamini, C. Galati and A. L. Knutsen, \emph{Moduli of nodal curves on K3 surfaces}, Adv. Math. \textbf{309}, 624--654 (2017)

\bibitem[Ch99]{[Ch99]}
X. Chen, \emph{Rational curves on K3 surfaces}, J. Alg. Geom. \textbf{8}, 245--278 (1999)

\bibitem[Ch02]{[Ch02]}
X. Chen, \emph{A simple proof that rational curves on K3 are nodal}, Math. Ann. \textbf{324} (1), 71--104 (2002)

\bibitem[Co07]{[Co07]}
B. Conrad, \emph{Arithmetic moduli of generalized elliptic curves}, J. Inst. Math. Jussieu \textbf{6} (2), 209--278 (2007) 

\bibitem[De20]{[De20]}
T. Dedieu, \emph{Comment on: On the irreducibility of the Severi variety of nodal curves in a smooth surface, by E. Ballico}, Arch. Math. \textbf{114}, 171--174 (2020)

\bibitem[DM69]{[DM69]}
P. Deligne and D. Mumford, \emph{The irreducibility of the space of curves of given genus}, Publ. Math. IH\'{E}S \textbf{36}, 75--109 (1969)

\bibitem[EGA III$\mathop{}_1$]{[EGAIII1]} A. Grothendieck, \emph{\'{E}l\'{e}ments de g\'{e}om\'{e}trie alg\'{e}brique: III. \'{E}tude cohomologique
des faisceaux coh\'{e}rents, Premi\`{e}re partie}, Publ. Math. IH\'{E}S. \textbf{11} (1961)

\bibitem[EGA III$\mathop{}_2$]{[EGAIII2]} A. Grothendieck, \emph{\'{E}l\'{e}ments de g\'{e}om\'{e}trie alg\'{e}brique: III. \'{E}tude cohomologique
des faisceaux coh\'{e}rents, Seconde partie}, Publ. Math. IH\'{E}S. \textbf{17} (1963)

\bibitem[FM94]{[FM94]}
W. Fulton and R. MacPherson, \emph{A compactification of configuration spaces}, Ann. Math. \textbf{139}, 183--225 (1994) 

\bibitem[FP97]{[FP95]}
W. Fulton and R. Pandharipande, \emph{Notes on stable maps and quantum cohomology}, in \emph{Algebraic Geometry} (Santa Cruz, 1995) edited by J. K\'{o}llar et. al., Proc. Sympos. Pure Math. \textbf{62}, 45--96, Amer. Math. Soc. (1997)

\bibitem[GK87]{[GK87]}
D. Gabai and W. Kazez, \emph{The classification of maps of surfaces}, Invent. Math. \textbf{90} (2), 219--242 (1987)

\bibitem[Ha77]{[Ha77]}
R. Hartshorne, \emph{Algebraic Geometry}, Graduate Texts in Mathematics \textbf{52}, Springer (1977)

\bibitem[Ha86]{[Ha86]}
J. Harris, \emph{On the Severi problem}, Invent. Math. \textbf{84} (3), 445--461 (1986)

\bibitem[HKK$\mathop{}^+$03]{[mirror-symmetry]}
K. Hori, S. Katz, A. Klemm, R. Pandharipande, R. Thomas, C. Vafa, R. Vakil, and E. Zaslow, \emph{Mirror symmetry}, Clay Mathematics Monographs (2003)

\bibitem[Ka03]{[Ka03]}
E. Kani, \emph{Hurwitz spaces of genus 2 covers of an elliptic curve}, Collect. Math. \textbf{54} (1), 1--51 (2003)

\bibitem[Ke05]{[Ke05]}
T. Keilen, \emph{Irreducibility of equisingular families of curves -- improved conditions}, Comm. Alg. \textbf{33}, 455--466 (2005)

\bibitem[Ke13]{[Ke13]}
M. Kemeny, \emph{The universal Severi variety of rational curves on K3 surfaces}, Bull. London Math. Soc. \textbf{45}, 159--174 (2013)

\bibitem[KKO14]{[KKO14]}
B. Kim, A. Kresch and Y.-G. Oh, \emph{A compactification of the space of maps from curves}, Trans. Amer. Math. Soc. \textbf{366}, 51--74 (2014)

\bibitem[Kl80]{[Kl80]}
S. Kleiman, \emph{Relative duality for quasi-coherent sheaves}, Compos. Math. \textbf{4} (1), 39--60 (1980)

\bibitem[KLM19]{[KLM19]}
A. L. Knutsen, M. Lelli-Chiesa and G. Mongardi, \textit{Severi varieties and Brill-Noether theory of curves on abelian surfaces}, J. Reine Angew. Math. \textbf{749} (4), 161--201 (2019)

\bibitem[KL19]{[KL19]}
A. L. Knutsen and M. Lelli-Chiesa, \textit{Genus two curves on abelian surfaces}, preprint, \href{https://arxiv.org/abs/1901.07603}{\texttt{arXiv:1901.07603}} (2019)

\bibitem[Kn83]{[Kn83]}
F.\ Knudsen, \emph{The projectivity of the moduli space of stable curves II: the stacks $\overline{\mathcal M}_{g,n}$}, Math. Scan. \textbf{52}, 161--199 (1983)

\bibitem[KT14]{[KT14]}
M. Kool and R. P. Thomas, \emph{Reduced classes and curve counting on surfaces I: theory}, Alg. Geom. \textbf{1}, 334--383 (2014)

\bibitem[Li01]{[Li01]}
J. Li, \emph{Stable morphisms to singular schemes and relative stable morphisms}, J. Diff. Geom. \textbf{57}, 509--578 (2001)

\bibitem[LLR04]{[LLR04]}
Q. Liu, D. Lorenzini, M. Raynaud, \emph{N\'{e}ron models, Lie algebras, and reduction of curves of genus one}, Invent. Math. \textbf{157}, 455--518 (2004)

\bibitem[Mu85]{[Mu85]}
D. Mumford, \emph{Abelian varieties}, 2nd reprinted edition, Tata Institute of Fundamental Research, Oxford Univ. Press (1985)

\bibitem[Ro11]{[Ro11]}
M. Romagny, \emph{Composantes connexes et irr\'{e}ductibles en familles}, Manuscripta Math. \textbf{136}, 1--32 (2011)

\bibitem[Stacks]{[stacks]}
The Stacks project authors, \emph{The Stacks project}, \url{https://stacks.math.columbia.edu/}

\bibitem[Se21]{[Se21]}
F. Severi, \emph{Vorlesungen \"{u}ber algebraische Geometrie}, Teubner (1921)

\bibitem[Te80]{[Te80]}
B. Teissier, \emph{R\'{e}solution simultan\'{e}e I, II}, Springer Lecture Notes in Mathematics \textbf{777}, 71--146 (1980)

\bibitem[Vo13]{[Vo13]}
C. Voisin, \emph{Hodge loci}, in \emph{Handbook of moduli: Volume III}, 507--546, International Press (2013)

\bibitem[Za19a]{[Za19a]}
A. Zahariuc, \emph{Elliptic surfaces and linear systems with fat points}, Math. Z. \textbf{293}, 647--660 (2019)

\bibitem[Za19b]{[Za19b]}
A. Zahariuc, \emph{The irreducibility of the generalized Severi varieties}, Proc. London Math. Soc. \textbf{119} (6), 1431--1463 (2019)
\end{thebibliography}
\end{document}